\documentclass[12pt]{amsart}

\usepackage{amscd, amsmath, amssymb, amsthm, amsfonts, amsxtra, enumerate, latexsym, tabularx, setspace}
\usepackage[all]{xy}
\usepackage[dvips]{graphicx,color}

\makeatletter
\@namedef{subjclassname@2020}{\textup{2020} Mathematics Subject Classification}
\makeatother

\def\Der{{\mathrm{Der}}}
\def\divisor{{\mathrm{div}}}
\def\Frac{{\mathrm{Frac}}}
\def\Proj{{\mathrm{Proj\; }}}
\def\Spec{{\mathrm{Spec\; }}}

\theoremstyle{plain}
\newtheorem{thm}{Theorem}[section]
\newtheorem*{theorem}{Main Theorem}
\newtheorem{crl}[thm]{Corollary}
\newtheorem{lem}[thm]{Lemma}

\theoremstyle{definition}

\newtheorem{exa}[thm]{Example}

\theoremstyle{remark}

\newtheorem*{acknowledgment}{Acknowledgments}

\title{Singular points configurations on quotients of the projective plane by $1$-foliations of degree $-1$ in characteristic $2$}
\author{Tadakazu Sawada}
\address{Department of General Education, National Institute of Technology, Fukushima College, 
30 Aza-Nagao, Kamiarakawa, Iwaki-shi, Fukushima 970-8034, Japan}
\email{sawada@fukushima-nct.ac.jp}


\begin{document}
\maketitle
\markboth{Tadakazu Sawada}{Singular points configurations on quotients of $\mathbb{P}^2$ by $1$-foliations of degree $-1$ in char $2$}

\begin{abstract}
In this paper, we classify the configurations of the singular points which appear on the quotients of the projective plane by the $1$-foliations of 
degree $-1$ in characteristic $2$. 
\end{abstract}

\section*{Introduction}
We work over an algebraically closed field $k$ of characteristic $p>0$. Let $X$ be a smooth algebraic variety and $F:X\rightarrow X$ 
be the Frobenius morphism of $X$. If a normal variety $Y$ decomposes the iterated Frobenius morphism as $F^e: X\rightarrow Y\rightarrow X$, 
we say that $Y$ is an $F^e$-sandwich of $X$. Since any purely inseparable morphism $f:X\rightarrow Y$ factors $F^e$ for some 
$e\geq 1$, the study of $F^e$-sandwiches leads to the study of purely inseparable morphisms. In particular, it is expected that the 
classifications of $F^e$-sandwiches play an important role in algebraic geometry in positive characteristics. 

Let $T_X$ be the tangent bundle of $X$. Given an invertible $1$-foliation $L\subset T_X$, see Section 1 for the definition of $1$-foliations, 
we have a Frobenius sandwich $X/L$ as the quotient of $X$ by $L$. In this paper, we consider the configurations of the singular points 
which appear on the quotients of the projective plane by the $1$-foliations $L$ of degree $-1$, $0$ and $1$ in characteristic $2$. In particular, 
we give the classification of the configurations of the case where $\deg L=-1$. 
\begin{theorem}[cf. Theorem~\ref{main}]
Let $Y$ be the quotient of the projective plane by an invertible $1$-foliation $L$ of degree $-1$. Then the configuration of the singular points 
of $Y$ coincides with the type $7A_1$, $D_4^0+3A_1$, $D_6^0+A_1$ or $E_7^0$. 
\end{theorem}
\noindent See Section 2 for the definition of the configuration types of the singular points. This result is a first step of an analysis of the 
purely inseparable coverings of the projective plane. 

In Section 1, we review the generalities of Frobenius sandwiches and fix a notation. In Section 2, we give the examples of the $F$-sandwiches 
of the affine plane and the projective plane in characteristic $2$. In Section 3, we consider the configurations of the singular points which appear 
on the quotients of the projective plane by the $1$-foliations of degree $-1$, $0$ and $1$ in characteristic $2$. 

In what follows, we adopt the notation of \cite{H}. 

\begin{acknowledgment}
This work is supported by JSPS KAKENHI Grant Number JP20K03526. 
\end{acknowledgment}

\section{Preliminaries}
First, we review the generalities of Frobenius sandwiches. Most of the following is citation from \cite{S4}. Let 
$X$ be an algebraic variety over $k$. The {\it absolute Frobenius morphism} $F: X \rightarrow X$ is the identity 
on the underlying topological space of $X$, and the $p$-th power map on the structure sheaf $\mathcal{O}_X$. 
Let $X^{(-1)}$ be the base change of $X$ by the absolute Frobenius morphism of $\Spec k$. The {\it relative 
Frobenius morphism} $F_{\rm rel}: X \rightarrow X^{(-1)}$ is defined by the following Cartesian square: 
$$\xymatrix
{X \ar@/^8mm/[rr]^F \ar[r]^{F_{\rm rel}} \ar[dr] & X^{(-1)} \ar[r] \ar[d] & X \ar[d] \\
                                                          & \Spec k \ar[r]^F     & \Spec k
}$$
In what follows, we do not strictly distinguish the absolute and relative Frobenius morphisms, since varieties are defined over the algebraically 
closed field $k$ in this paper. 

Let $X$ be a smooth variety over $k$. A normal variety $Y$ is an {\it $F^e$-sandwich} of $X$ if the $e$-th iterated 
relative Frobenius morphism of $X$ is decomposed as 
$$\xymatrix
{X \ar[rr]^{F^e_{\rm rel}} \ar[rd]_{\pi}  &                  & X^{(-e)} \\
                                          & Y \ar[ru]_{\rho} &
}$$
for some finite $k$-morphisms $\pi : X \rightarrow Y$ and $\rho : Y \rightarrow X^{(-e)}$, which are homeomorphisms 
in the Zariski topology. An {\it $F$-sandwich} means an $F^1$-sandwich. 
We say that 
\begin{itemize}
\item a variety $Y$ is a {\it Frobenius sandwich} of $X$ if $Y$ is an $F^e$-sandwich of $X$ for some $e\geq 1$, 
\item an $F^e$-sandwich $X\xrightarrow{\pi}Y\xrightarrow{\rho} X^{(-e)}$ is {\it of degree $p$} if the degree of the morphism $\pi$ is $p$, and 
\item $F^e$-sandwiches $X\xrightarrow{\pi}Y\xrightarrow{\rho} X^{(-e)}$ and $X\xrightarrow{\pi'}Y'\xrightarrow{\rho'} X^{(-e)}$ 
are {\it equivalent} if there exists an isomorphism $f:Y\rightarrow Y'$ such that $f\circ \pi=\pi'$ and $\rho =\rho' \circ f$: 
$$\xymatrix
{X \ar[r]^{F^e_{\rm rel}} \ar[d]_{\pi} \ar[rd]^(.34){\!\!\!\pi'} & X^{(-e)} \\
 Y \ar[ru]_(.34){\!\!\!\rho} \ar[r]_{f} & Y' \ar[u]_{\rho'}
}$$
\end{itemize}

Let $R$ be a $k$-algebra and $\Der_k R$ be the set of all $k$-derivations on $R$. A derivation $\delta\in \Der_k R$ 
is {\it $p$-closed} if $\delta^p=f\delta$ for some $f\in R$, where $\delta^p$ is a $p$-times iterated composite as a 
differential operator. By a {\it $1$-foliation} of $X$, we mean a saturated $p$-closed subsheaf $L$ of the tangent 
bundle $T_X$ closed under the Lie brackets, where $L$ is said to be {\it $p$-closed} if $\delta^p\in L$ for all $\delta\in L$. 
We say that $1$-foliations $L$ and $L'$ are {\it equivalent} if there exists an isomorphism $f$ such that the inclusion 
morphism $L\hookrightarrow T_X$ is factored as $L\xrightarrow{f}L'\hookrightarrow T_X$: 
$$\xymatrix
{L \ar@{^{(}->}[r] \ar[d]_{f}  & T_X \\
 L' \ar@{^{(}->}[r] & T_X \ar@{=}[u]
}$$

It is known that there are one-to-one correspondences among 
\begin{itemize}
\item the equivalence classes of the $F$-sandwiches of $X$ of degree $p$, 
\item the equivalence classes of the invertible $1$-foliations of $X$, and
\item the equivalence classes of the non-zero $p$-closed rational vector fields of $X$. 
\end{itemize}
In what follows, we quickly review these correspondences. See \cite{RS}, \cite{E} and \cite{Hirokado} for details. 

\noindent (1) $F$-sandwiches and $1$-foliations: The correspondence is given by
$$Y \mapsto L=\{\delta\in T_X\,|\,\delta(f)=0\text{ for all }f\in \mathcal{O}_Y\}\subset T_X$$
and
$$L \mapsto 
\mathcal{O}_Y=\{f\in\mathcal{O}_X\,|\,\delta(f)=0\text{ for all }\delta\in L\}\subset \mathcal{O}_X.$$
The well-definedness of this correspondence is guaranteed by the Galois correspondence due to Aramova and Avramov 
\cite{AA}. 

\noindent (2) $F$-sandwiches and rational vector fields: A {\it rational vector field} means an element of 
$\Der_k\, K(X)$, where $K(X)$ is the function field of $X$. We define an equivalence relation $\sim$ between 
rational vector fields $\delta$ and $\delta'$ as follows: $\delta \sim {\delta}'$ if and only if there exists a 
non-zero rational function $\alpha \in K(X)$ such that $\delta = \alpha {\delta}'$. Let $\{U_i = \Spec R_i\}_i$ 
be an affine open covering of $X$. Given a non-zero $p$-closed rational vector field $\delta \in \Der_k\, K(X)$, 
we have the quotient variety $X/\delta$ defined by glueing $\Spec R_i^{\delta}$, where 
$R_i^{\delta}=\{r\in R_i\,|\,\delta(r)=0\}$, and the quotient map $\pi_{\delta} : X \rightarrow X/\delta$ induced 
from the inclusions $R_i^{\delta} \subset R_i$. We see that $R^{\delta}_i$ is a normal domain and the field extension 
$\Frac\, R_i/\Frac\, R_i^{\delta}$ is purely inseparable of degree $p$. This means that $X/\delta$ is an 
$F$-sandwich of degree $p$ with the finite morphism $\pi_{\delta} : X \rightarrow X/\delta$ through which the 
Frobenius morphism of $X$ is decomposed. Conversely, if $Y$ is an $F$-sandwich of $X$ of degree $p$ with a finite 
morphism $\pi : X \rightarrow Y$ through which the Frobenius morphism of $X$ is decomposed, then there exists a 
non-zero $p$-closed rational vector field $\delta\in \Der_k\, K(X)$ such that $\pi = \pi_{\delta}$ and $Y = X/{\delta}$. 
Indeed, there exists a non-zero $p$-closed rational vector field $\delta \in \Der_k\, K(X)$ such that 
$K(X)^{\delta}=K(Y)$ by Baer's result, since the field extension $K(X)/K(Y)$ is purely inseparable of degree 
$p$. Then $\delta$ induces an inclusion $\mathcal{O}_Y \subset \mathcal{O}_{X/{\delta}}$, so that there exists a 
finite birational morphism $X/{\delta} \rightarrow Y$. Since $Y$ is normal, this morphism is an isomorphism. See 
\cite{J} for Baer's result. 

\noindent (3) $1$-foliations and rational vector fields: A non-zero $p$-closed rational vector field 
$\delta \in \Der_k\, K(X)$ is locally expressed as $\alpha \sum f_i \partial/\partial s_i$, where $s_i$ are local 
coordinates, $f_i$ are regular functions without common factors, and $\alpha \in K(X)$. The {\it divisor 
$\divisor\, \delta$ associated to $\delta$} is defined by glueing the divisors $\divisor\, \alpha$. The injective 
morphism $\cdot\, \delta: \mathcal{O}_X(\divisor\, \delta)\rightarrow T_X$ defined locally by 
$h/\alpha \mapsto h(\sum f_i \partial/\partial s_i)$, where $h$ is a regular function, gives 
$\mathcal{O}_X(\divisor\, \delta)$ the structure of an invertible $1$-foliation. We see that the correspondence 
$\delta \mapsto \mathcal{O}_X(\mathrm{div}(\delta))$ gives the one-to-one correspondence between 
the equivalence classes of the non-zero $p$-closed rational vector fields and the equivalence classes of 
the invertible $1$-foliations. 

Let $X$ be a smooth surface and $\delta \in \Der_k\, K(X)$ be a non-zero $p$-closed rational vector 
field of $X$. Suppose that $\delta$ is locally expressed as 
$\delta=\alpha ( f {\partial}/{\partial s} + g {\partial}/{\partial t})$, where $s$ and $t$ are local coordinates, 
$f$ and $g$ are regular functions without common factors, and $\alpha \in K(X)$. Then the singular points 
of the $F$-sandwich $Y=X/\delta$ 
lie on the images 
of the points on $X$ defined by the equations $f=g=0$. 

Next, we fix a notation about the projective plane $\mathbb{P}^2$. Let $X_0$, $X_1$ and $X_2$ be the homogeneous 
coordinates of $\mathbb{P}^2$, i.e., we assume that $\mathbb{P}^2 = \Proj k [X_0, X_1, X_2]$. Let $x=X_1/X_0$ 
and $y=X_2/X_0$ (resp. $z=X_0/X_1$ and $w=X_2/X_1$ ; $u=X_0/X_2$ and $v=X_1/X_2$) be the affine coordinates of 
$U_0 := D_+ (X_0)=\Spec [X_1/X_0, X_2/X_0]$ (resp. $U_1 := D_+ (X_1)=\Spec [X_0/X_1, X_2/X_1]$ ; 
$U_2 := D_+ (X_2)=\Spec [X_0/X_2, X_1/X_2]$). A symbol $\partial_s$ denotes a partial derivation $\partial/\partial_s$ for each 
variable $s=x$, $y$, $z$, $w$, $u$ and $v$. We have $x={1}/{z}={v}/{u}$ and $y={w}/{z}={1}/{u}$, so that 
$\partial_x = - z^2 \partial_z - zw \partial_w = u \partial_v$ and 
$\partial_y = z \partial_w = -u^2 \partial_u - uv \partial_v$. 

\section{Examples}
In this section, we assume that the characteristic $p=2$ and follow the notation about the projective plane 
$\mathbb{P}^2$ of the end of Section 1. See \cite{A} for the notation of the rational double points in positive 
characteristics. 

Let $R$ be a $k$-algebra. A {\it regular vector field of $\Spec R$} means an element of $\Der_k\, R$. 
Let $\mathbb{A}^2=\Spec k[x,y]$ be the affine plane and $\delta \in \Der_k\, k[x,y]$ be a non-zero $p$-closed regular 
vector field of $\mathbb{A}^2$. Then the Frobenius sandwich $Y$ of $\mathbb{A}^2$ as the quotient by $\delta$ is described 
as $Y=\Spec k[x,y]^{\delta}$, where $k[x,y]^{\delta}=\{f\in k[x,y]\,|\,\delta (f) =0\}$. First, we give the examples of 
the singular points which appear on the $F$-sandwiches $Y$ of $\mathbb{A}^2$ of degree $p$ as the quotients by non-zero $p$-closed regular 
vector fields $\delta$. 
\begin{exa} 
Suppose that $a$ and $b$ are non-zero elements of $k$ and $n$ is an integer. 
\begin{enumerate}[\normalfont \rmfamily (1)]
\item $\delta=x\partial_x+y\partial_y$: We have $k[x,y]^{\delta}=k[x^2,y^2,xy]\cong k[X,Y,Z]/(Z^2+XY)$. Thus $Y$ 
has an $A_1$-singularity at the image of the origin of $\mathbb{A}^2$. 
\item $\delta=x(x+a)\partial_x+ay\partial_y$: We have $k[x,y]^{\delta}=k[x^2,y^2,x(x+a)y]\cong k[X,Y,Z]/(Z^2+X(X+a^2)Y)$. 
Let $R=k[X,Y,Z]/(Z^2+X(X+a^2)Y)$. In consideration of the localizations of $R$ by the maximal ideals $(X,Y,Z)R$ and 
$(X+a^2,Y,Z)R$, we see that $Y$ has two $A_1$-singularities at the images of the origin and $(a,0)$ of $\mathbb{A}^2$. 
\item $\delta=x\partial_x+y(ay+1)\partial_y$: As with the case (2), we see that $Y$ has two $A_1$-singularities at 
the images of the origin and $(0,a^{-1})$ of $\mathbb{A}^2$. 
\item $\delta=x(x+a)\partial_x+y(by+a)\partial_y$: As with the case (2), we see that $Y$ has four $A_1$-singularities at 
the images of the origin, $(0, ab^{-1})$, ($a$, 0) and $(a, ab^{-1})$ of $\mathbb{A}^2$. 
\item $\delta=(x^2+anxy^{n-1})\partial_x+ay^n\partial_y\ (n\geq 2)$: We have $k[x,y]^{\delta}=k[x^2,y^2,x^2y+axy^n]
\cong k[x^2,y^2,x^2y+xy^n]\cong k[X,Y,Z]/(Z^2+X^2Y+XY^n)$. Thus $Y$ has a $D_{2n}^0$-singularity at the image of the 
origin of $\mathbb{A}^2$. 
\item $\delta=xy^2\partial_x+(ax^2+y^3)\partial_y$: We have $k[x,y]^{\delta}=k[x^2,y^2,ax^3+xy^3]\cong k[x^2,y^2,x^3+xy^3]
\cong k[X,Y,Z]/(Z^2+X^3+XY^3)$. Thus $Y$ has an $E_7^0$-singularity at the image of the origin of $\mathbb{A}^2$. 
\end{enumerate}

We also see that an $E_8^0$-singularity appears on $Y$ for $\delta=y^4\partial_x+ax^2\partial_y$. The following table 
summarizes these examples. In the table, $nA_1$ stands for the case where $Y$ has $n$ points $A_1$-singularities. 
$$\begin{array}{c|c}
\mbox{non-zero $p$-closed regular vector field $\delta$} & \mbox{singular points which appear on $Y=\mathbb{A}^2/\delta$}\\ \hline
x\partial_x+y\partial_y & A_1 \\ \hline
x(x+a)\partial_x+ay\partial_y, x\partial_x+y(ay+1)\partial_y & 2A_1 \\ \hline
x(x+a)\partial_x+y(by+a)\partial_y & 4A_1 \\ \hline
(x^2+anxy^{n-1})\partial_x+ay^n\partial_y\ (n\geq 2) & D_{2n}^0 \\ \hline
xy^2\partial_x+(ax^2+y^3)\partial_y & E_7^0 \\ \hline
y^4\partial_x+ax^2\partial_y & E_8^0
\end{array}$$
\end{exa}

Next, we give the examples of the singular points which appear on the $F$-sandwiches $Y$ of $\mathbb{P}^2$ of degree $p$ as the quotients by non-zero 
$p$-closed rational vector fields $\delta \in \Der_k\, k(x,y)$. In what follows, $\mathcal{O}_X(\divisor\, \delta)$ is written as $L$ and 
$\pi_{\delta}$ is abbreviated to $\pi$. 
\begin{exa}\label{example-2}
Suppose that $a$ and $b$ are non-zero elements of $k$. 
\begin{enumerate}[\normalfont \rmfamily (1)]
\item $\delta=x\partial_x+y\partial_y$: 
We have 
$$\delta=x\partial_x+y\partial_y=z\partial_z=u\partial_u.$$
Thus $\deg L=1$. Moreover we see that $Y$ has an $A_1$-singularity at the image of the origin of $U_0$ and is smooth 
on the images $\pi(U_1)$ and $\pi(U_2)$. Since $\pi:\mathbb{P}^2\rightarrow Y$ is homeomorphic on the underlying topological spaces, we may identify 
$\mathbb{P}^2$ and $Y$ through $\pi$ as topological spaces. Under this identification, we can figure out the configuration 
of the singular points of $Y$ as follows: 

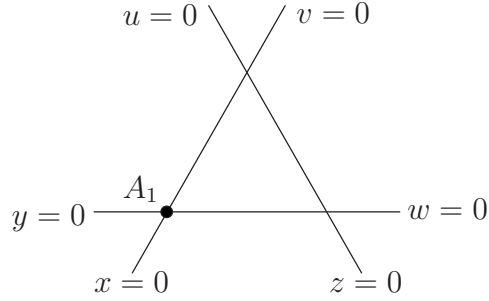
\centerline{
\unitlength 0.1in
\begin{picture}( 16.0000, 18.0000)(  4.0000,-18.0000)
%
{\color[named]{Black}{%
\special{pn 8}%
\special{pa 600 1600}%
\special{pa 1400 200}%
\special{fp}%
}}%
%
{\color[named]{Black}{%
\special{pn 8}%
\special{pa 1000 200}%
\special{pa 1800 1600}%
\special{fp}%
}}%
%
{\color[named]{Black}{%
\special{pn 8}%
\special{pa 400 1280}%
\special{pa 2000 1280}%
\special{fp}%
}}%
%
{\color[named]{Black}{%
\special{pn 0}%
\special{sh 1.000}%
\special{ia 780 1280 30 30  0.0000000 6.2831853}%
}}%
{\color[named]{Black}{%
\special{pn 8}%
\special{ar 780 1280 30 30  0.0000000 6.2831853}%
}}%
\put(5.5000,-12.3000){\makebox(0,0)[lb]{$A_1$}}%
\put(4.0000,-17.0000){\makebox(0,0)[lb]{$x=0$}}
\put(-0.3000,-13.8000){\makebox(0,0)[lb]{$y=0$}}
\put(16.3000,-16.9800){\makebox(0,0)[lb]{$z=0$}}
\put(20.4000,-13.2000){\makebox(0,0)[lb]{$w=0$}}
\put(5.5000,-3.0000){\makebox(0,0)[lb]{$u=0$}}
\put(14.6000,-2.9300){\makebox(0,0)[lb]{$v=0$}}
\end{picture}%
}

\noindent If we can settle the singular points as this configuration by suitable coordinate changes, we say that the 
configuration of the singular points of $Y$ is the {\it type $A_1$}. 

In what follows, we omit the representations of $x=0$, $y=0$, $z=0$, $w=0$, $u=0$ and $v=0$ in the figures of the 
configuration of the singular points. 

\item $\delta=x^2\partial_x+ay^2\partial_y$: We have 
$$\delta=x^2\partial_x+ay^2\partial_y=\dfrac{1}{z}(z\partial_z+w(aw+1)\partial_w)=\dfrac{a}{u}(u\partial_u+v(a^{-1}v+1)\partial_v).$$
Thus $\deg L=-1$. Moreover we see that $Y$ has a $D_4^0$-singularity at the image of the origin of $U_0$ and 
two $A_1$-singularities on $\pi(U_1)$ and $\pi(U_2)$. We can figure out the configuration of the singular points 
of $Y$ as follows: 

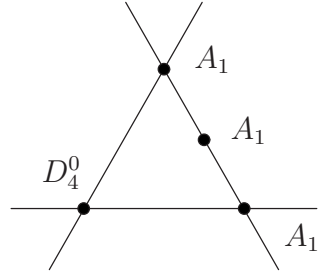
\centerline{
\unitlength 0.1in
\begin{picture}( 16.0000, 18.0000)(  4.0000,-18.0000)
%
{\color[named]{Black}{%
\special{pn 8}%
\special{pa 600 1600}%
\special{pa 1400 200}%
\special{fp}%
}}%
%
{\color[named]{Black}{%
\special{pn 8}%
\special{pa 1000 200}%
\special{pa 1800 1600}%
\special{fp}%
}}%
%
{\color[named]{Black}{%
\special{pn 8}%
\special{pa 400 1280}%
\special{pa 2000 1280}%
\special{fp}%
}}%
%
{\color[named]{Black}{%
\special{pn 0}%
\special{sh 1.000}%
\special{ia 780 1280 30 30  0.0000000 6.2831853}%
}}%
{\color[named]{Black}{%
\special{pn 8}%
\special{ar 780 1280 30 30  0.0000000 6.2831853}%
}}%
%
{\color[named]{Black}{%
\special{pn 0}%
\special{sh 1.000}%
\special{ia 1620 1280 30 30  0.0000000 6.2831853}%
}}%
{\color[named]{Black}{%
\special{pn 8}%
\special{ar 1620 1280 30 30  0.0000000 6.2831853}%
}}%
%
{\color[named]{Black}{%
\special{pn 0}%
\special{sh 1.000}%
\special{ia 1200 550 30 30  0.0000000 6.2831853}%
}}%
{\color[named]{Black}{%
\special{pn 8}%
\special{ar 1200 550 30 30  0.0000000 6.2831853}%
}}%
%
{\color[named]{Black}{%
\special{pn 0}%
\special{sh 1.000}%
\special{ia 1410 920 30 30  0.0000000 6.2831853}%
}}%
{\color[named]{Black}{%
\special{pn 8}%
\special{ar 1410 920 30 30  0.0000000 6.2831853}%
}}%
\put(5.6000,-11.8000){\makebox(0,0)[lb]{$D_4^0$}}%
\put(13.6000,-5.7000){\makebox(0,0)[lb]{$A_1$}}%
\put(15.5000,-9.4000){\makebox(0,0)[lb]{$A_1$}}%
\put(18.3000,-15.0000){\makebox(0,0)[lb]{$A_1$}}%
\end{picture}%
}

\noindent If we can settle the singular points as this configuration by suitable coordinate changes, we say that the 
configuration of the singular points of $Y$ is the {\it type $D_4^0+3A_1$}. 

\item $\delta=x(x+a)\partial_x+ay\partial_y$: We have 
$$\delta =x(x+ a) d_x + ay d_y=\dfrac{a}{z}(z(z+a^{-1})d_z+a^{-1}wd_w)=\dfrac{a}{u}(u^2d_u+a^{-1}v^2d_v).$$
Thus $\deg L= -1$ and the configuration of the singular points of $Y$ is the type $D_4^0+3A_1$. This holds also 
for $\delta=x\partial_x+y(ay+1)\partial_y$. 

\item $\delta=x(x+a)\partial_x+y(by+a)\partial_y$: We have 
\begin{eqnarray*}
\delta &=& x(x+ a) d_x + y(by + a) d_y\\
&=&\dfrac{a}{z}(z(z+a^{-1})d_z+w(a^{-1}bw+a^{-1})d_w)=\dfrac{a}{u}(u(u+a^{-1}b)d_u+v(a^{-1}v+a^{-1}b)d_v).
\end{eqnarray*}
Thus $\deg L=-1$. Moreover we see that $Y$ has four $A_1$-singularities on $\pi(U_0)$, $\pi(U_1)$ and 
$\pi(U_2)$. We can figure out the configuration of the singular points of $Y$ as follows: 

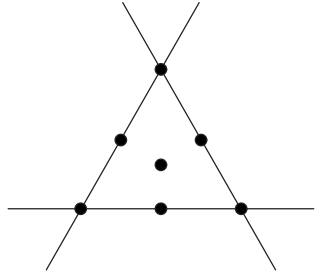
\centerline{
\unitlength 0.1in
\begin{picture}( 16.0000, 18.0000)(  4.0000,-18.0000)
%
{\color[named]{Black}{%
\special{pn 8}%
\special{pa 600 1600}%
\special{pa 1400 200}%
\special{fp}%
}}%
%
{\color[named]{Black}{%
\special{pn 8}%
\special{pa 1000 200}%
\special{pa 1800 1600}%
\special{fp}%
}}%
%
{\color[named]{Black}{%
\special{pn 8}%
\special{pa 400 1280}%
\special{pa 2000 1280}%
\special{fp}%
}}%
%
{\color[named]{Black}{%
\special{pn 0}%
\special{sh 1.000}%
\special{ia 780 1280 30 30  0.0000000 6.2831853}%
}}%
{\color[named]{Black}{%
\special{pn 8}%
\special{ar 780 1280 30 30  0.0000000 6.2831853}%
}}%
%
{\color[named]{Black}{%
\special{pn 0}%
\special{sh 1.000}%
\special{ia 1620 1280 30 30  0.0000000 6.2831853}%
}}%
{\color[named]{Black}{%
\special{pn 8}%
\special{ar 1620 1280 30 30  0.0000000 6.2831853}%
}}%
{\color[named]{Black}{%
\special{pn 0}%
\special{sh 1.000}%
\special{ia 1200 1280 30 30  0.0000000 6.2831853}%
}}%
{\color[named]{Black}{%
\special{pn 8}%
\special{ar 1200 1280 30 30  0.0000000 6.2831853}%
}}%
%
{\color[named]{Black}{%
\special{pn 0}%
\special{sh 1.000}%
\special{ia 1200 550 30 30  0.0000000 6.2831853}%
}}%
{\color[named]{Black}{%
\special{pn 8}%
\special{ar 1200 550 30 30  0.0000000 6.2831853}%
}}%
%
{\color[named]{Black}{%
\special{pn 0}%
\special{sh 1.000}%
\special{ia 1410 920 30 30  0.0000000 6.2831853}%
}}%
{\color[named]{Black}{%
\special{pn 8}%
\special{ar 1410 920 30 30  0.0000000 6.2831853}%
}}%
{\color[named]{Black}{%
\special{pn 0}%
\special{sh 1.000}%
\special{ia 1200 1050 30 30  0.0000000 6.2831853}%
}}%
{\color[named]{Black}{%
\special{pn 8}%
\special{ar 1200 1050 30 30  0.0000000 6.2831853}%
}}%
{\color[named]{Black}{%
\special{pn 0}%
\special{sh 1.000}%
\special{ia 990 920 30 30  0.0000000 6.2831853}%
}}%
{\color[named]{Black}{%
\special{pn 8}%
\special{ar 990 920 30 30  0.0000000 6.2831853}%
}}%
\end{picture}%
}

\noindent In the figure, the solid circles $\bullet$ stand for $A_1$-singularities. 
If we can settle the singular points as this configuration by suitable coordinate changes, we say that the configuration 
of the singular points of $Y$ is the {\it type $7A_1$}. 

\item $\delta=(x^2+axy^2)\partial_x+ay^3\partial_y$: 
We have 
$$\delta=(x^2+axy^2)\partial_x+ay^3\partial_y=\dfrac{1}{z}((z+aw^2)\partial_z+w\partial_w)=\dfrac{1}{u}(a\partial_u+v^2\partial_v).$$
Thus $\deg L=-1$ and $Y$ is smooth on $\pi(U_2)$. On $\pi(U_0)$, we have $k[x,y]^{\delta}=k[x^2,y^2,x^2y+axy^3]
\cong k[x^2,y^2,x^2y+xy^3]\cong k[X,Y,Z]/(Z^2+X^2Y+XY^3)$. On $\pi(U_1)$, we have $k[z,w]^{\delta}=k[z^2,w^2,zw+aw^3]
=k[z^2,w^2,(z+aw^2)w]\cong k[z^2,w^2,zw]\cong k[X,Y,Z]/(Z^2+XY)$. Therefore $Y$ has a $D_6^0$-singularity at the 
image of the origin of $U_0$ and an $A_1$-singularity at the image of the origin of $U_1$. We can figure out the 
configuration of the singular points of $Y$ as follows: 

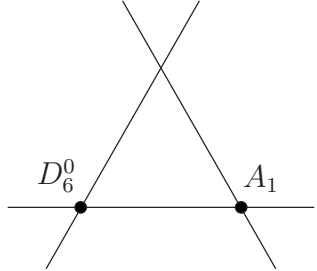
\centerline{
\unitlength 0.1in
\begin{picture}( 16.0000, 18.0000)(  4.0000,-18.0000)
%
{\color[named]{Black}{%
\special{pn 8}%
\special{pa 600 1600}%
\special{pa 1400 200}%
\special{fp}%
}}%
%
{\color[named]{Black}{%
\special{pn 8}%
\special{pa 1000 200}%
\special{pa 1800 1600}%
\special{fp}%
}}%
%
{\color[named]{Black}{%
\special{pn 8}%
\special{pa 400 1280}%
\special{pa 2000 1280}%
\special{fp}%
}}%
%
{\color[named]{Black}{%
\special{pn 0}%
\special{sh 1.000}%
\special{ia 780 1280 30 30  0.0000000 6.2831853}%
}}%
{\color[named]{Black}{%
\special{pn 8}%
\special{ar 780 1280 30 30  0.0000000 6.2831853}%
}}%
%
{\color[named]{Black}{%
\special{pn 0}%
\special{sh 1.000}%
\special{ia 1620 1280 30 30  0.0000000 6.2831853}%
}}%
{\color[named]{Black}{%
\special{pn 8}%
\special{ar 1620 1280 30 30  0.0000000 6.2831853}%
}}%
\put(5.5000,-12.0000){\makebox(0,0)[lb]{$D_6^0$}}%
\put(16.3000,-12.0000){\makebox(0,0)[lb]{$A_1$}}%
\end{picture}%
}

\noindent If we can settle the singular points as this configuration by suitable coordinate changes, we say that the 
configuration of the singular points of $Y$ is the {\it type $D_6^0+A_1$}. 

\item $\delta=xy^2\partial_x+(ax^2+y^3)\partial_y$: We have 
$$\delta=xy^2\partial_x+(ax^2+y^3)\partial_y=\dfrac{1}{z}(w^2\partial_z+a\partial_w)=\dfrac{1}{u}((auv^2+1)\partial_u+av^3\partial_v).$$
Thus $\deg L=-1$. Moreover we see that $Y$ has an $E_7^0$-singularity at the image of the origin of $U_0$ and 
is smooth on $\pi(U_1)$ and $\pi(U_2)$. We can figure out the configuration of the singular points of $Y$ as follows: 

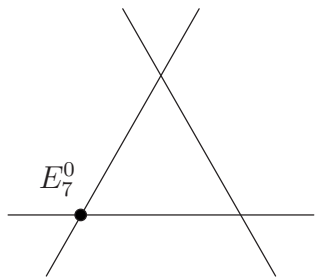
\centerline{
\unitlength 0.1in
\begin{picture}( 16.0000, 18.0000)(  4.0000,-18.0000)
%
{\color[named]{Black}{%
\special{pn 8}%
\special{pa 600 1600}%
\special{pa 1400 200}%
\special{fp}%
}}%
%
{\color[named]{Black}{%
\special{pn 8}%
\special{pa 1000 200}%
\special{pa 1800 1600}%
\special{fp}%
}}%
%
{\color[named]{Black}{%
\special{pn 8}%
\special{pa 400 1280}%
\special{pa 2000 1280}%
\special{fp}%
}}%
%
{\color[named]{Black}{%
\special{pn 0}%
\special{sh 1.000}%
\special{ia 780 1280 30 30  0.0000000 6.2831853}%
}}%
{\color[named]{Black}{%
\special{pn 8}%
\special{ar 780 1280 30 30  0.0000000 6.2831853}%
}}%
\put(5.6000,-11.8000){\makebox(0,0)[lb]{$E_7^0$}}%
\end{picture}%
}

\noindent If we can settle the singular points as this configuration by suitable coordinate changes, we say that the 
configuration of the singular points of $Y$ is the {\it type $E_7^0$}. 
\end{enumerate}
\end{exa}

\section{Main theorem}
Also in this section, we assume that the characteristic $p=2$ and follow the notation about the projective plane 
$\mathbb{P}^2$ of the end of Section 1. 

Let $Y$ be the $F$-sandwich of $\mathbb{P}^2$ of degree $p$, $L\subset T_{\mathbb{P}^2}$ (resp. $\delta\in 
\Der_k\,K(\mathbb{P}^2)$) be the corresponding invertible $1$-foliation (resp. the non-zero $p$-closed rational vector 
field). We have an exact sequence $0\rightarrow L\rightarrow T_{\mathbb{P}^2}\rightarrow I_Z\otimes L' \rightarrow 0$, 
where $Z$ is a $0$-dimensional subscheme whose (topological) image coincides with the support of the singular 
points of $Y$, $I_Z$ is an ideal sheaf of $Z$, and $L'$ is an invertible sheaf. The following fact is reffed in \cite{GR}. 
For the convenience of readers we give proof here. 

\begin{lem}[Ganong-Russell\cite{GR}]\label{1}
Let $H\subset \mathbb{P}^2$ be a line. If $c_1(L)=nH$, then $c_2(I_Z)=n^2-3n+3$. 
\end{lem}
\begin{proof}
By the Euler sequence $0\rightarrow \mathcal{O}_{\mathbb{P}^2}\rightarrow \mathcal{O}(1)^{\oplus 3} \rightarrow T_{\mathbb{P}^2} \rightarrow 0$, 
we have $c_1(\mathcal{O}_{\mathbb{P}^2}(1)^{\oplus 3})=c_1(\mathcal{O}_{\mathbb{P}^2})+c_1(T_{\mathbb{P}^2})=c_1(T_{\mathbb{P}^2})$ 
and $c_2(\mathcal{O}_{\mathbb{P}^2}(1)^{\oplus 3})=c_2(T_{\mathbb{P}^2})+c_1(\mathcal{O}_{\mathbb{P}^2})c_1(T_{\mathbb{P}^2})=c_2(T_{\mathbb{P}^2})$. 
Thus we have $c_1(T_{\mathbb{P}^2})=3H$ and $c_2(T_{\mathbb{P}^2})=3$. Let $c_1(L)=nH$. Since 
$c_1(T_{\mathbb{P}^2})=c_1(L)+c_1(I_Z\otimes L')=c_1(L)+c_1(L')$, we have $c_1(L')=(3-n)H$. Since 
$c_2(T_X)=c_1(L)c_1(I_Z\otimes L')+c_2(I_Z\otimes L')=c_1(L)c_1(L')+c_2(I_Z)=n(3-n)+c_2(I_Z)$, we see that $c_2(I_Z)=n^2-3n+3$. 
\end{proof}

Note that $n=\deg L\leq 1$, since $T_{\mathbb{P}^2}$ is stable. For any integer $n\leq 1$, we see that $n^2-3n+3\not= 0$. 
For example, if $\deg c_1(L)=1$ {\rm (resp. $0; -1$)}, then $c_2(I_Z)=1$ {\rm (resp. 3; 7)}. Thus we have the following:

\begin{crl}[Bloch, Ganong-Russell\cite{GR}]
The F-sandwiches of $\mathbb{P}^2$ of degree $p$ are singular. 
\end{crl}

Note that this Corollary holds in any positive characteristic. Indeed, our proof of Lemma~\ref{1} does not use the 
assumption that $p=2$. 

Does there exist the case where $Y$ has $m$ singular points for every $m$ such that $1\leq m\leq n^2-3n+3$? In particular, 
does there exist the case where $Y$ has only one or $n^2-3n+3$ singular points? In the following, we classify the configurations 
of the singular points of $Y$ and give the answers of these questions in the case where $n=-1$. 

\begin{lem}\label{lem-of-der}
Let $L$ be an invertible $1$-foliation of $\mathbb{P}^2$ such that $\deg L\geq -1$. After suitable coordinate changes, the 
corresponding non-zero $p$-closed rational vector field is equivalent to $\delta=F\partial_x+G\partial_y$ such that
\begin{itemize}
\item[(i)] $F=a_{30}x^3+a_{12}xy^2+a_{20}x^2 + a_{02}y^2 + a_{10}x$ and $G=a_{30}x^3+a_{12}xy^2+a_{20}x^2 + a_{02}y^2 + a_{10}x$ for some $a_{ij}, b_{ij}\in k$, 
\item[(ii)] $F\not=0$ and $G\not=0$, and 
\item[(iii)] $F$ and $G$ have no common factors. 
\end{itemize}
\end{lem}
\begin{proof}
In what follows, $f_x$ (resp. $f_y$) means $\partial f/\partial x$ (resp. $\partial f/\partial y$) and $\deg_x f$ (resp. $\deg_y f$) 
stands for the degree of $f$ with respect to the variable $x$ (resp. $y$) for $f\in k[x,y]$. 

Since there exists at least one singular point on the corresponding $F$-sandwich $Y=\mathbb{P}^2/\delta$, we 
may assume that a singular point lies on the image of the origin of $U_0$ after suitable coordinate changes. Then, by multiplying 
a suitable rational function, we may assume that $\delta=F\partial_x+G\partial_y$ where $F=\sum_{1\leq i,j}a_{ij}x^iy^j$ and 
$G=\sum_{1\leq i,j}b_{ij}x^iy^j$ for some $a_{ij}, b_{ij}\in k$. Since $\delta\not=0$, we have $F\not=0$ or $G\not=0$. If $F=0$ and 
$G\not=0$ (resp. $F\not= 0$ and $G=0$), then $\delta$ is equivalent to $\partial_y$ (resp. $\partial_x$), from which it follows 
that $Y$ is smooth at the image of the origin of $U_0$. Thus $F\not=0$ and $G\not=0$, which is the condition (ii). By multiplying 
a suitable rational function, we may assume also that $F$ and $G$ have no common factors, which is the condition (iii). 

To see that the condition (i) holds, we first show that $F=a_{30}x^3+ a_{21}x^2y+a_{12}xy^2+a_{20}x^2 + a_{11}xy + a_{02}y^2 + a_{10}x + a_{01}y$ 
and $G=a_{30}x^2y+a_{21}xy^2+a_{12}y^3+b_{20}x^2 + b_{11}xy + b_{02}y^2 + b_{10}x + b_{01}y$. By expressing $\delta$ for the local 
coordinates $z$ and $w$, we have 
\begin{eqnarray*}
\delta &=& \sum a_{ij} \frac{1}{z^i} \frac{w^j}{z^j} \left( - z^2 \partial_z - zw \partial_w \right)
 + \sum b_{mn} \frac{1}{z^m} \frac{w^n}{z^n} z \partial_w 
 = - \sum a_{ij} \frac{w^j}{z^{i+j-2}} \partial_z + \widetilde{G} \partial_w, 
\end{eqnarray*}
where $\widetilde{G}\in k\left[z, w, 1/z\right]$. Since $F$ and $G$ have no common factors, we see that 
$\deg L=\deg \mathcal{O}_X(\divisor\,\delta)$ is estimated by the common factor $1/z^k$ of the coefficients of $\partial_z$ and 
$\partial_w$. Since $\deg L \geq -1$, we have $a_{ij}=0$ for $i+j \geq 4$. In the same way, by expressing $\delta$ for the local 
coordinates $u$ and $v$, we see that $b_{mn}=0$ for $m+n \geq 4$. Thus we have
\begin{eqnarray*}
\delta &=& \sum_{1 \leq i+j \leq 3} a_{ij}x^iy^j \partial_x + \sum_{1 \leq m+n \leq 3} b_{mn}x^my^n \partial_y \\
&=& (a_{30}/z+ a_{21}w/z+a_{12}w^2/z+a_{03}w^3/z +a_{20} + a_{11} w + a_{02} w^2 + a_{10} z + a_{01} zw) \partial_z \\
& & + (b_{30}/z^2+ (a_{30}+b_{21})w/z^2+(a_{21}+b_{12})w^2/z^2+(a_{12}+b_{03})w^3/z^3 + a_{03}w^4/z^2 \\
& & +b_{20}/z + (a_{20}+b_{11})w/z + (a_{11}+b_{02})w^2/z + a_{02}w^3/z 
+ b_{10} + (a_{10}+b_{01}) w +a_{01} w^2) \partial_w. 
\end{eqnarray*}
Since $\deg L \geq -1$, we see that $b_{30}=0$, $b_{21}=a_{30}$, $b_{12}=a_{21}$, $b_{03}=a_{12}$ and $a_{03}=0$. Therefore 
$F=a_{30}x^3+ a_{21}x^2y+a_{12}xy^2+a_{20}x^2 + a_{11}xy + a_{02}y^2 + a_{10}x + a_{01}y$ and 
$G=a_{30}x^2y+a_{21}xy^2+a_{12}y^3+b_{20}x^2 + b_{11}xy + b_{02}y^2 + b_{10}x + b_{01}y$. 

We next show that $F=a_{30}x^3+a_{12}xy^2+a_{20}x^2 + a_{02}y^2 + a_{10}x$ and $G=a_{30}x^2y+a_{12}y^3+b_{20}x^2+ b_{02}y^2+ a_{10}y$. 
Suppose that $\delta^2=H\delta$, where $H\in k[x,y]$. Since $\delta^2=(FF_x+GF_y) \partial_x+(FG_x+GG_y)\partial_y$, we have 
$FF_x+GF_y=HF$ and $FG_x+GG_y=HG$. Assume that $F_y\not=0$. Since $F$ and $G$ have no common factors and $GF_y=F(F_x+H)$, 
we see that $F_y=IF$ for some $I\in k[x,y]$. Since $\deg_x F\leq \deg_xF_y=\deg_x(a_{21}x^2+a_{11}x+a_{01})=2$ and 
$\deg_y F\leq \deg_y F_y=\deg_y(a_{21}x^2+a_{11}x+a_{01})=0$, we have $F=a_{20}x^2 + a_{10}x$. Then $F_y=0$, which contradicts to the 
assumption that $F_y\not=0$. Thus $F_y=a_{21}x^2+a_{11}x+a_{01}=0$. In the same way, we see that $G_x=a_{21}y^2+b_{11}y+b_{10}=0$, 
so that $a_{21}=a_{11}=a_{01}=b_{11}=b_{10}=0$. Now, we have 
$\delta^2=FF_x \partial_x+GG_y\partial_y=(a_{30}x^2+a_{12}y^2+a_{10})F\partial_x+(a_{30}x^2+a_{12}y^2+b_{01})G\partial_y$. Since $\delta$ 
is $p$-closed, we also see that $a_{10}=b_{01}$. Therefore $F=a_{30}x^3+a_{12}xy^2+a_{20}x^2 + a_{02}y^2 + a_{10}x$ and 
$G=a_{30}x^2y+a_{12}y^3+b_{20}x^2+ b_{02}y^2+ a_{10}y$. 
\end{proof}

\begin{thm}\label{main}
Let $Y$ be the quotient of $\mathbb{P}^2$ by an invertible $1$-foliation $L$. 
\begin{enumerate}[\normalfont \rmfamily (1)]
\item If $\deg L=1$, then the configuration of the singular points of $Y$ coincides with the type $A_1$ after suitable coordinate changes. 
\item An invertible $1$-foliation $L$ of degree $0$ dose not exist. 
\item If $\deg L=-1$, then the configuration of the singular points of $Y$ coincides with the type $7A_1$, $D_4^0+3A_1$, $D_6^0+A_1$ or $E_7^0$. 
\end{enumerate}
\end{thm}
\begin{proof}
Let $\delta \in \Der_k\,K(\mathbb{P}^2)$ be the non-zero $p$-closed rational vector field corresponding to $L$. By Lemma~\ref{lem-of-der} 
we may assume that $\delta=F\partial_x+G\partial_y$ such that 
\begin{itemize}
\item[(i)] $F=a_{30}x^3+a_{12}xy^2+a_{20}x^2 + a_{02}y^2 + a_{10}x$ and $G=a_{30}x^3+a_{12}xy^2+a_{20}x^2 + a_{02}y^2 + a_{10}x$ for some $a_{ij}, b_{ij}\in k$, 
\item[(ii)] $F\not=0$ and $G\not=0$, and 
\item[(iii)] $F$ and $G$ have no common factors. 
\end{itemize}
We have
\begin{eqnarray*}
\delta &=& (a_{30}x^3+a_{12}xy^2+a_{20}x^2  + a_{02}y^2 + a_{10}x ) \partial_x + (a_{30}x^2y+a_{12}y^3+b_{20}x^2  + b_{02}y^2 + a_{10}y) \partial_y \\
&=& \dfrac{1}{z}((a_{02} zw^2+a_{10} z^2 +a_{12}w^2+a_{20}z+a_{30}) \partial_z + ( a_{02}w^3+ b_{02}w^2+ a_{20}w+b_{20}) \partial_w) \\
&=& \dfrac{1}{u}((b_{20} uv^2 +a_{10} u^2+a_{30}v^2+b_{02}u+a_{12}) \partial_u +( b_{20}v^3+ a_{20} v^2 + b_{02} v + a_{02}) \partial_v). 
\end{eqnarray*}
We divide proof into the three cases: 
\begin{itemize}
\item[(A)] One can settle the singular points of $Y$ on the images of the origins of $U_0$, $U_1$ and $U_2$ by suitable coordinate changes of $\mathbb{P}^2$. 
\item[(B)] One can settle the singular points of $Y$ on the images of the origins of $U_0$ and $U_1$ by suitable coordinate changes of $\mathbb{P}^2$ 
and cannot put the singular points as the case (A). 
\item[(C)] One can settle the singular point of $Y$ on the image of the origin of $U_0$ by suitable coordinate changes 
and cannot put the singular points as the case (A) and (B), i.e., there exists the only one singular point on $Y$. 
\end{itemize} 
In what follows, 
\begin{itemize}
\item the logical connectives $\wedge$ (and) and $\vee$ (or) are used to list up multiple conditions, 
\item $\pi_{\delta}$ is abbreviated to $\pi$, and 
\item $\sqrt{\alpha}$ stands for the unique root of the equation $x^2+\alpha=0$ for $\alpha\in k$. 
\end{itemize}

\noindent(A) We settle the singular points on the images of the origins of $U_0$, $U_1$ and $U_2$ by suitable coordinate changes. 
Then we may assume that $a_{30}=b_{20}=a_{12}=a_{02}=0$, so that 
\begin{eqnarray*}
\delta &=& (a_{20}x^2+ a_{10}x ) \partial_x + ( b_{02}y^2 + a_{10}y) \partial_y \\
&=& \dfrac{1}{z}((a_{10} z^2+a_{20}z) \partial_z + ( b_{02} w^2+a_{20} w) \partial_w) = \dfrac{1}{u}(( a_{10} u^2+b_{02}u ) \partial_u +( a_{20} v^2 +b_{02} v) \partial_v).
\end{eqnarray*} 
We consider the eight cases such that each coefficient $a_{20}$, $a_{10}$ and $b_{02}$ is zero or not. The following table 
summarizes what we know. In the table, $0$ in the column of $a_{20}$ (resp. $a_{10}$; $b_{02}$) mean that $a_{20}=0$ (resp. 
$a_{10}=0$; $b_{02}=0$). On the other hand, $k^{\ast}$ in the column of $a_{20}$ (resp. $a_{10}$; $b_{02}$) mean that 
$a_{20}\not=0$ (resp. $a_{10}\not=0$; $b_{02}\not=0$). We use this notation also in the tables which will come out later.\\
$$\begin{array}{c|c|c|c|c|c} 
 & a_{20} & a_{10} & b_{02} & \deg L & \\ \hline 
(\heartsuit) & 0 & 0 & 0& & F=G=0.\\ \hline
(\heartsuit) & 0 & 0 & k^{\ast} & & F=0.\\ \hline
(1) & 0 & k^{\ast} & 0& 1 & A_1 \\ \hline
(2) & 0 & k^{\ast} & k^{\ast} & -1 & D_4^0+3A_1\\ \hline
(\heartsuit) & k^{\ast} & 0 & 0& & G=0.\\ \hline
(3) & k^{\ast} & 0 & k^{\ast} & -1 & D_4^0+3A_1\\ \hline
(4) & k^{\ast} & k^{\ast} & 0& -1 & D_4^0+3A_1\\ \hline
(5) & k^{\ast} & k^{\ast} & k^{\ast} & -1 & 7A_3
\end{array}$$
The cases ($\heartsuit$) contradict to the assumption (ii). \\
(1) $\delta = a_{10}x  \partial_x + a_{10}y \partial_y$: By multiplying $a_{10}^{-1}$, we may assume that 
$$\delta = x  \partial_x + y \partial_y= z  \partial_z= u \partial_u.$$
This is the case (1) in Example~\ref{example-2}, so that $\deg L=1$ and the configuration of the singular points of $Y$ is the 
type $A_1$. \\
(2) $\delta = a_{10}x  \partial_x + ( b_{02}y^2 + a_{10}y) \partial_y$: 
By multiplying $a_{10}^{-1}$, we may assume that 
$$\delta = x \partial_x + y( b_{02}y + 1) \partial_y=\dfrac{1}{z}(z^2\partial_z+b_{02}w^2\partial_w)=\dfrac{1}{u}(u(u+b_{02})\partial_u+b_{02}v\partial_v).$$
This is the case (3) in Example~\ref{example-2}, so that $\deg L=-1$ and the configuration of the singular points of $Y$ is the 
type $D_4^0+3A_1$. \\
(3) $\delta= a_{20}x^2 \partial_x + b_{02}y^2 \partial_y$: By multiplying $a_{20}^{-1}$, we may assume that 
$$\delta = x^2 \partial_x + b_{02}y^2 \partial_y=\dfrac{1}{z}(z\partial_z+w(b_{02}w+1))\partial_w=\dfrac{1}{u}(b_{02}u\partial_u+v(v+b_{02})\partial_v).$$
This is the case (2) in Example~\ref{example-2}, so that $\deg L=-1$ and the configuration of the singular points of $Y$ is the 
type $D_4^0+3A_1$. \\
(4) $\delta=(a_{20}x^2+ a_{10}x ) \partial_x + a_{10}y \partial_y$: By multiplying $a_{20}^{-1}$, we may assume that 
$$\delta =x(x+a_{10}) \partial_x + a_{10}y \partial_y=\dfrac{a_{10}}{z}(z(z+a_{10}^{-1})d_z+a_{10}^{-1}wd_w)=\dfrac{a_{10}}{u}(u^2d_u+a_{10}^{-1}v^2d_v).$$
This is the case (3) in Example~\ref{example-2}, so that $\deg L=-1$ and the configuration of the singular points of $Y$ is the 
type $D_4^0+3A_1$. \\
(5) $\delta=(a_{20}x^2+ a_{10}x ) \partial_x + ( b_{02}y^2 + a_{10}y) \partial_y$: By multiplying $a_{20}^{-1}$, we may assume that 
\begin{eqnarray*}
\delta &=& x(x+a_{10} ) \partial_x + y( b_{02}y + a_{10}) \partial_y\\
 &=&\dfrac{a_{10}}{z}(z(z+a_{10}^{-1})d_z+w(a_{10}^{-1}b_{02}w+a^{-1})d_w)=\dfrac{a_{10}}{u}(u(u+a_{10}^{-1}b_{02})d_u+v(a_{10}^{-1}v+a_{10}^{-1}b_{02})d_v).
\end{eqnarray*}
This is the case (4) in Example~\ref{example-2}, so that $\deg L=-1$ and the configuration of the singular points of $Y$ is the 
type $7A_1$. \\
\noindent(B) We settle the singular points on the images of the origins of $U_0$ and $U_1$ by suitable coordinate changes. 
Then we have $(a_{30}=b_{20}=0)\wedge (a_{12}\not=0 \vee a_{02}\not=0)$ and 
\begin{eqnarray*}
\delta &=& (a_{12}xy^2+a_{20}x^2  + a_{02}y^2 + a_{10}x ) \partial_x + (a_{12}y^3+ b_{02}y^2 + a_{10}y) \partial_y \\
&=& \dfrac{1}{z}((a_{02} zw^2+a_{10} z^2 +a_{12}w^2+a_{20}z) \partial_z + ( a_{02}w^3+ b_{02}w^2+ a_{20}w) \partial_w) \\
&=& \dfrac{1}{u}((a_{10} u^2+b_{02}u+a_{12}) \partial_u +(a_{20} v^2 + b_{02} v + a_{02}) \partial_v). 
\end{eqnarray*}
\noindent(B-1) Suppose that $a_{12}=0$. We have $a_{02}\not=0$. By multiplying $a_{02}^{-1}$, we may assume that 
\begin{eqnarray*}
\delta &=& (a_{20}x^2  + y^2 + a_{10}x ) \partial_x + ( b_{02}y^2 + a_{10}y) \partial_y \\
 &=& \dfrac{1}{z}((a_{10} z^2 +  zw^2 + a_{20}z ) \partial_z + (  w^3 + b_{02}w^2 + a_{20}w ) \partial_w)\\
 &=&  \dfrac{1}{u}((a_{10} u^2 + b_{02}u) \partial_u +(  a_{20} v^2 + b_{02} v + 1) \partial_v). 
\end{eqnarray*}
We consider the eight cases such that each coefficient $a_{20}$, $a_{10}$ and $b_{02}$ is zero or not. The following table summarizes 
what we know. 
$$\begin{array}{c|c|c|c|c|c} 
 & a_{20} & a_{10} & b_{02} & \deg L & \\ \hline 
(\heartsuit) & 0 & 0 & 0 & & G=0 \\ \hline
(\clubsuit) & 0 & 0 & k^{\ast} & & \delta = y^2 \partial_x + b_{02}y^2 \partial_y \\ \hline
(1) & 0 & k^{\ast} & 0 & -1 & D_6^0+A_1\hspace{1mm} \\ \hline
(2) & 0 & k^{\ast} & k^{\ast} & & (A) \\ \hline
(\heartsuit) & k^{\ast} & 0 & 0 & & G=0 \\ \hline
(3) & k^{\ast} & 0 & k^{\ast} & & (A) \\ \hline
(4) & k^{\ast} & k^{\ast} & 0 & & (A) \\ \hline
(5) & k^{\ast} & k^{\ast} & k^{\ast} & & (A) 
\end{array}$$
The cases ($\heartsuit$) and ($\clubsuit$) contradict to the assumption (ii) and (iii), respectively. \\
(1) We have
$$\delta =( y^2 + a_{10}x )\partial_x + a_{10}y\partial_y = \dfrac{1}{z}(z(w^2+a_{10}z)\partial_z+w^3\partial_w) = \dfrac{1}{u}(a_{10}u^2\partial_u+\partial_v).$$
We see that $\deg L=-1$ and $Y$ is smooth on $\pi(U_2)$. On $\pi(U_0)$, we have $k[x^2, y^2, y^3+a_{10}xy]= k[x^2, y^2, (x+a_{10}y^2)y] 
\cong k[x^2, y^2, xy] \cong k[X,Y,Z]/(Z^2+XY)$. On $\pi(U_1)$, we have $k[z^2, w^2, zw^3+az^2w]\cong k[z^2, w^2, z^2w+zw^3]
\cong k[X,Y,Z]/(Z^2+X^2Y+XY^3)$. Thus the configuration of the singular points of $Y$ is the type $D_6^0+A_1$. \\
(2) $\delta = (y^2 + a_{10}x) \partial_x + y(b_{02}y + a_{10}) \partial_y$: Since 
$F(0,0)=G(0,0)=F(a_{10}b_{02}^{-2}, a_{10}b_{02}^{-1})=G(a_{10}b_{02}^{-2}, a_{10}b_{02}^{-1})=0$, we see that $Y$ has singular points at the 
images of the origin and $(a_{10}b_{02}^{-2}, a_{10}b_{02}^{-1})$ on $U_0$. Since $(a_{10}b_{02}^{-2}, a_{10}b_{02}^{-1})$ does not lie on the 
line defined by $X_2=0$, we can settle the singular points on the images of the origins of $U_0$, $U_1$ and $U_2$ by suitable coordinate 
changes. This is the case (A). \\
(3) We have
\begin{eqnarray*}
\delta &=& (a_{20}x^2 + y^2 ) \partial_x  + b_{02}y^2 \partial_y\\
&=& \dfrac{1}{z}(z(w^2+a_{20}) \partial_z + w(w^2+b_{02}w+a_{20}) \partial_w)= \dfrac{1}{u}(b_{02}u \partial_u + (a_{20}v^2 + b_{02}v + 1) \partial_v).
\end{eqnarray*}
Let $F_1=z(w^2+a_{20})$ and $G_1=w(w^2+b_{02}w+a_{20})$. Since $(w^2+b_{02}w+a_{20})|_{w=0}=a_{20}\not=0$, the equation $w^2+b_{02}w+a_{20}=0$ 
has a non-zero root $w=\alpha$. Since $F_1(0,0)=G_1(0,0)=F_1(0,\alpha)=G_1(0,\alpha)=0$, we see that $Y$ has singular points at the 
images of the origin and $(0,\alpha)$ on $U_1$. Since $(0,\alpha)$ lies on the line defined by $X_0=0$, we can settle the singular points on the 
images of the origins of $U_0$, $U_1$ and $U_2$ by suitable coordinate changes. This is the case (A). \\
(4) We have 
\begin{eqnarray*}
\delta &=& (a_{20}x^2 + y^2 + a_{10}x ) \partial_x + a_{10}y \partial_y \\
&=& \dfrac{1}{z}(z(w^2 +a_{10}z+ a_{20} ) \partial_z + w(w^2+a_{20}) \partial_w)= \dfrac{1}{u}(a_{10} u^2 \partial_u +(a_{20}v^2 + 1) \partial_v). 
\end{eqnarray*}
Let $F_1=z(w^2 +a_{10}z+ a_{20})$ and $G_1=w(w^2+a_{20})$. Since $F_1(0,0)=G_1(0,0)=F_1(0,\sqrt{a_{20}})=G_1(0,\sqrt{a_{20}})=0$, we see that 
$Y$ has singular points at the images of the origin and $(0,\sqrt{a_{20}})$ on $U_1$. Since $(0,\sqrt{a_{20}})$ lies on the line defined by $X_0=0$, 
we can settle the singular points on the images of the origins of $U_0$, $U_1$ and $U_2$ by suitable coordinate changes. This is the case (A). \\
(5) $\delta=(a_{20}x^2 + y^2 + a_{10}x ) \partial_x + y(b_{02}y + a_{10}) \partial_y$: We see that $y=a_{10}b_{02}^{-1}$ is a root of the equation 
$y(b_{02}y + a_{10})=0$. Since $a_{10}^2b_{02}^{-2}\not=0$, the equation $F(x, a_{10}b_{02}^{-1})=a_{20}x^2+a_{10}x+a_{10}^2b_{02}^{-2}=0$ has a 
non-zero root $x=\alpha$. Since $F(0,0)=G(0,0)=F(\alpha,a_{10}b_{02}^{-1})=G(\alpha,a_{10}b_{02}^{-1})=0$, we see that $Y$ has singular points 
at the images of the origin and $(\alpha,a_{10}b_{20}^{-1})$ on $U_0$. Since $(\alpha,a_{10}b_{20}^{-1})$ does not lie on the line defined by $X_2=0$, 
we can settle the singular points on the images of the origins of $U_0$, $U_1$ and $U_2$ by suitable coordinate changes. This is the case (A). \\
\noindent(B-2) Suppose that $a_{12}\not=0$. By multiplying $a_{12}^{-1}$, we may assume that 
\begin{eqnarray*}
\delta &=& (xy^2+a_{20}x^2  + a_{02}y^2 + a_{10}x ) \partial_x  + (y^3  + b_{02}y^2 + a_{10}y) \partial_y \\
&=& \dfrac{1}{z}(\left(w^2 +a_{20}z  + a_{02} zw^2 + a_{10} z^2 \right) \partial_z + \left(  a_{20} w + b_{02}w^2 + a_{02}w^3 \right) \partial_w) \\
&=& \dfrac{1}{u}(\left( 1 + b_{02}u  + a_{10} u^2\right) \partial_u +\left( a_{20}v^2 +b_{02}v + a_{02} \right) \partial_v).
\end{eqnarray*}
We consider the sixteen cases such that each coefficient $a_{20}$, $a_{02}$, $a_{10}$ and $b_{02}$ is zero or not. The following table 
summarizes what we know. 
$$\begin{array}{c|c|c|c|c|c|c} 
 & a_{20} & a_{02} & a_{10} & b_{02} & \deg L & \\ \hline 
(\clubsuit) & 0 &0 &0 &0 & & \delta=xy^2 \partial_x  + y^3 \partial_y \\ \hline 
(\clubsuit) & 0 &0 &0 & k^{\ast} & & \delta=xy^2 \partial_x  + y^2(y + b_{02}) \partial_y \\ \hline
(\clubsuit) & 0 &0 & k^{\ast} &0 & & \delta=x(y^2+ a_{10} ) \partial_x  + y(y^2 + a_{10}) \partial_y \\ \hline
(1) & 0 &0 & k^{\ast} & k^{\ast} & & (A) \\ \hline
(\clubsuit) & 0 & k^{\ast} &0 &0 & & \delta=y^2(x + a_{02}) \partial_x  + y^3  \partial_y\\ \hline
(\clubsuit) & 0 & k^{\ast} &0 &k^{\ast} & & \delta=y^2(x + a_{02}) \partial_x  + y^2(y + b_{02}) \partial_y\\ \hline
(2) & 0 & k^{\ast} & k^{\ast} &0 & -1 & D_6^0+A_1 \\ \hline
(3) & 0 & k^{\ast} & k^{\ast} & k^{\ast} & & (A) \\ \hline
(4) & k^{\ast} &0 &0 &0 & -1 & D_6^0+A_1 \\ \hline
(5) & k^{\ast} &0 &0 & k^{\ast} & & (A) \\ \hline
(6) & k^{\ast} &0 & k^{\ast} &0 & & (A) \\ \hline
(7) & k^{\ast} &0 & k^{\ast} & k^{\ast} & & (A) \\ \hline
(8) & k^{\ast} & k^{\ast} &0 &0 & -1 & D_6^0+A_1 \\ \hline
(9) & k^{\ast} & k^{\ast} &0 & k^{\ast} & & (A) \\ \hline
(10) & k^{\ast} & k^{\ast} & k^{\ast} &0 & & (A) \\ \hline
(11) & k^{\ast} & k^{\ast} & k^{\ast} & k^{\ast} & & (A) 
\end{array}$$
The cases ($\clubsuit$) contradict to the assumption (iii). \\
(1) $\delta=x(y^2+a_{10}) \partial_x  + y(y^2  + b_{02}y + a_{10}) \partial_y $: Since $(y^2+b_{02}y+a_{10})|_{y=0}=a_{10}\not= 0$, the equation 
$y^2+b_{02}y+a_{10}=0$ has a non-zero root $y=\alpha$. We see that $Y$ has singular points at the images of the origin and $(0, \alpha)$ 
on $U_0$. Since $(0,\alpha)$ lies on the line defined by $X_1=0$, we can settle the singular points on the images of the origins of $U_0$, $U_1$ 
and $U_2$ by suitable coordinate changes. This is the case (A). \\
(2) We have 
\begin{eqnarray*}
\delta &=& (xy^2+ a_{02}y^2 + a_{10}x) \partial_x + (y^3  + a_{10}y) \partial_y \\
&=& \dfrac{1}{z}((a_{02}zw^2+ a_{10} z^2+w^2 ) \partial_z + a_{02}w^3 \partial_w) = \dfrac{1}{u}((a_{10}u^2+1) \partial_u +a_{02} \partial_v). 
\end{eqnarray*}
We see that $\deg L=-1$ and $Y$ is smooth on $\pi(U_2)$. On $U_0$, we have $k[x,y]^{\delta}=k[x^2, y^2, xy^3+a_{02}y^3+a_{10}xy]
\cong k[X, Y, Z]/(Z^2+XY^3+a_{02}^2Y^3+a_{10}^2XY)= k[X, Y, Z]/(Z^2+(X+a_{02}^2)Y^3+a_{10}^2XY)$. Let 
$R=k[X, Y, Z]/(Z^2+(X+a_{02}^2)Y^3+a_{10}^2XY)$ and $\mathfrak{m}=(X,Y,Z)R$. The $\mathfrak{m}$-adic completion of $R$ is isomorphic 
to $k[[X, Y, Z]]/(Z^2+(X+a_{02}^2)Y^3+a_{10}^2XY)\cong k[[X, Y, Z]]/(Z^2+Y^3+XY)= k[[X, Y, Z]]/(Z^2+(X+Y^2)Y)\cong k[[X, Y, Z]]/(Z^2+XY)$. 
On $U_1$, we have $k[z,w]^{\delta}=k[z^2, w^2, a_{02}zw^3+a_{10}z^2w+w^3]=k[z^2, w^2, a_{02}zw^3+a_{10}z^2w+w^3+ (a_{02}\sqrt{a_{10}^{-1}})w^4]
=k[z^2, w^2, (\sqrt{a_{10}}z+w)^2w+(a_{02}\sqrt{a_{10}^{-1}})(\sqrt{a_{10}}z+w)w^3]\cong k[z^2, w^2, z^2w+(a_{02}\sqrt{a_{10}^{-1}})zw^3]
\cong k[z^2, w^2, z^2w+zw^3]\cong k[X,Y,Z]/(Z^2+X^2Y+XY^3)$. Thus the configuration of the singular points of $Y$ is the type $D_6^0+A_1$. \\
(3) $\delta=(xy^2+ a_{02}y^2 + a_{10}x ) \partial_x  +y (y^2 + b_{02}y + a_{10}) \partial_y$: Since $(y^2+b_{02}y+a_{10})|_{y=0}=a_{10}\not= 0$, 
the equation $y^2+a_{02}y+a_{10}=0$ has a non-zero root $y=\alpha$. We see that $x=a_{02}\alpha^2(\alpha^2+a_{10})^{-1}$ is the root of the 
equation $F(x,\alpha)=\alpha^2x+a_{02}\alpha^2+a_{10}x=0$ and $Y$ has singular points at the images of the origin and 
$(a_{02}\alpha^2(\alpha^2+a_{10})^{-1}, \alpha)$ on $U_0$. Since $(a_{02}\alpha^2(\alpha^2+a_{10})^{-1}, \alpha)$ does not lie on the line 
defined by $X_2=0$, we can settle the singular points on the images of the origins of $U_0$, $U_1$ and $U_2$ by suitable coordinate changes. 
This is the case (A). \\
(4) We have 
$$\delta = (xy^2+a_{20}x^2 ) \partial_x  + y^3 \partial_y = \dfrac{1}{z}((w^2 +a_{20}z ) \partial_z + a_{20}w \partial_w)= \dfrac{1}{u}(\partial_u + a_{20}v^2 \partial_v).$$
We see that $\deg L=-1$ and $Y$ is smooth on $\pi(U_2)$. On $U_0$, we have $k[x,y]^{\delta}=k[x^2, y^2, xy^3+a_{20}x^2y]\cong k[x^2, y^2, xy^3+x^2y] 
\cong k[X,Y,Z]/(Z^2+XY^3+X^2Y) = k[X,Y,Z]/(Z^2+X^2Y+XY^3)$. On $U_1$, we have $k[z,w]^{\delta}=k[z^2, w^2, w^3+a_{20}zw]=k[z^2, w^2, (a_{20}z+w^2)w]
\cong k[z^2, w^2, zw]\cong k[X,Y,Z]/(Z^2+XY)$. Thus the configuration of the singular points of $Y$ is the type $D_6^0+A_1$. \\
(5) $\delta=x(y^2 + a_{20}x ) \partial_x  + y^2(y + b_{02}) \partial_y$: We see that $Y$ has singular points at the images of the origin and 
$(0,b_{02})$ on $U_0$. Since $(0,b_{02})$ lies on the line defined by $X_1=0$, we can settle the singular points on the images of the origins of $U_0$, 
$U_1$ and $U_2$ by suitable coordinate changes. This is the case (A). \\
(6) $\delta=x(y^2+a_{20}x + a_{10} ) \partial_x  +y (y^2+ a_{10}) \partial_y$: We see that $Y$ has singular points at the images of the origin and 
$(0,\sqrt{a_{10}})$ on $U_0$. Since $(0,\sqrt{a_{10}})$ lies on the line defined by $X_1=0$, we can settle the singular points on the images of the origins 
of $U_0$, $U_1$ and $U_2$ by suitable coordinate changes. This is the case (A). \\
(7) $\delta=x(y^2+a_{20}x + a_{10}) \partial_x  +y (y^2 + b_{02}y + a_{10}) \partial_y$: Since $(y^2+b_{02}y+a_{10})|_{y=0}=a_{10}\not= 0$, the equation 
$y^2+b_{02}y+a_{10}=0$ has a non-zero root $y=\alpha$. We see that $Y$ has singular points at the images of the origin and $(0, \alpha)$ on $U_0$. 
Since $(0,\alpha)$ lies on the line defined by $X_1=0$, we can settle the singular points on the images of the origins of $U_0$, $U_1$ and $U_2$ by suitable 
coordinate changes. This is the case (A). \\
(8) We have 
\begin{eqnarray*}
\delta &=& (xy^2+a_{20}x^2  + a_{02}y^2 ) \partial_x + y^3 \partial_y \\
&=& \dfrac{1}{z}((a_{02}zw^2+w^2 +a_{20}z ) \partial_z + (a_{02}w^3+a_{20}w) \partial_w)= \dfrac{1}{u}( \partial_u + (a_{20} v^2 + a_{02}) \partial_v).
\end{eqnarray*}
We see that $\deg L=-1$ and $Y$ is smooth on $\pi(U_2)$. On $U_0$, we have $k[x,y]^{\delta}=k[x^2, y^2, xy^3+a_{20}x^2y+a_{02}y^3]
=k[x^2, y^2, xy^3+a_{20}x^2y+a_{02}y^3+\sqrt{a_{02}}\sqrt{a_{20}^{-1}}y^4]
=k[x^2, y^2, (\sqrt{a_{20}}x+\sqrt{a_{02}}y)^2y+\sqrt{a_{20}^{-1}}(\sqrt{a_{20}}x+\sqrt{a_{02}}y)y^3]
\cong k[x^2, y^2, x^2y+\sqrt{a_{20}^{-1}}xy^3]\cong k[x^2, y^2, x^2y+xy^3]\cong k[X,Y,Z]/(Z^2+X^2Y+XY^3)$. On $U_1$, we have 
$k[z, w]^{\delta}=k[z^2, w^2, a_{02}zw^3+w^3+a_{20}zw]\cong k[X,Y,Z]/(Z^2+a_{02}^2XY^3+Y^3+a_{20}^2XY)=k[X,Y,Z]/(Z^2+(a_{02}^2X+1)Y^3+a_{20}^2XY)$. 
Let $R=k[X,Y,Z]/(Z^2+(a_{02}^2X+1)Y^3+a_{20}^2XY)$ and $\mathfrak{m}=(X,Y,Z)R$. The $\mathfrak{m}$-adic completion of $R$ is isomorphic to 
$k[[X,Y,Z]]/(Z^2+(a_{02}^2X+1)Y^3+a_{20}^2XY)\cong k[[X,Y,Z]]/(Z^2+Y^3+XY)=k[[X,Y,Z]]/(Z^2+(X+Y^2)Y)\cong k[[X,Y,Z]]/(Z^2+XY)$. Thus the 
configuration of the singular points of $Y$ is the type $D_6^0+A_1$. \\
(9) $\delta=(xy^2+a_{20}x^2 + a_{02}y^2) \partial_x  + y^2(y + b_{02}) \partial_y$: Since $(a_{20}x^2+b_{02}^2x+a_{02}b_{02}^2)|_{x=0}=a_{02}b_{02}^2\not= 0$, 
the equation $F(x,b_{02})=a_{20}x^2+b_{02}^2x+a_{02}b_{02}^2=0$ has a non-zero root $x=\alpha$. We see that $Y$ has singular points at the images of 
the origin and $(\alpha, b_{02})$ on $U_0$. Since $(\alpha, b_{02})$ does not lie on the line defined by $X_2=0$, we can settle the singular points on the 
images of the origins of $U_0$, $U_1$ and $U_2$ by suitable coordinate changes. This is the case (A). \\
(10) $\delta=(xy^2+a_{20}x^2 + a_{02}y^2 + a_{10}x ) \partial_x  +y (y^2+ a_{10}) \partial_y$: We see that $x=\sqrt{a_{02}a_{10}a_{20}^{-1}}$ is the non-zero 
root of the equation $F(x,\sqrt{a_{10}})=a_{20}x^2 + a_{02}a_{10}=0$ and $Y$ has singular points at the images of the origin and 
$(\sqrt{a_{02}a_{10}a_{20}^{-1}}, \sqrt{a_{10}})$ on $U_0$. Since $(\sqrt{a_{02}a_{10}a_{20}^{-1}},\sqrt{a_{10}})$ does not lie on the line defined by $X_2=0$, 
we can settle the singular points on the images of the origins of $U_0$, $U_1$ and $U_2$ by suitable coordinate changes. This is the case (A). \\
(11) $\delta=(xy^2+a_{20}x^2 + a_{02}y^2 + a_{10}x ) \partial_x  +y (y^2 + b_{02}y + a_{10}) \partial_y$: Since $(y^2+b_{02}y+a_{10})|_{y=0}=a_{10}\not= 0$, 
the equation $y^2+b_{02}y+a_{10}=0$ has a non-zero root $y=\alpha$. Since $(a_{20}x^2+(a_{10}+\alpha^2)x+a_{02}\alpha^2)|_{x=0}=a_{02}\alpha^2\not= 0$, 
the equation $F(x,\alpha)=a_{20}x^2+(a_{10}+\alpha^2)x+a_{02}\alpha^2=0$ has a non-zero root $x=\beta$. We see that $Y$ has singular points at the 
images of the origin and $(\beta, \alpha)$ on $U_0$. Since $(\beta, \alpha)$ does not lie on the line defined by $X_2=0$, we can settle the singular points 
on the images of the origins of $U_0$, $U_1$ and $U_2$ by suitable coordinate changes. This is the case (A). \\
\noindent(C) We settle the singular point on the image of the origin of $U_0$ by suitable coordinate changes. We may assume that 
\begin{eqnarray*}
\delta &=& (a_{30}x^3+a_{12}xy^2+a_{20}x^2 + a_{02}y^2 + a_{10}x ) \partial_x  + (a_{30}x^2y+a_{12}y^3+b_{20}x^2 + b_{02}y^2 + a_{10}y) \partial_y \\
 &=& \dfrac{1}{z}((a_{02} zw^2+a_{10} z^2+a_{12}w^2 +a_{20}z + a_{30}) \partial_z + ( a_{02}w^3+b_{02}w^2  + a_{20}w +b_{20}  ) \partial_w)\\
 &=& \dfrac{1}{u}((b_{20} uv^2+a_{10} u^2+a_{30}v^2 + b_{02}u+a_{12} ) \partial_u +( b_{20}v^3+ a_{20}v^2 + b_{02}v + a_{02} ) \partial_v).
\end{eqnarray*}
Since $Y$ is smooth on $\pi(U_1)$ and $\pi(U_2)$, we see that $(a_{30}\not=0 \vee b_{20}\not=0) \wedge (a_{12}\not=0 \vee a_{02}\not=0)$. \\
(C-1) Suppose that $a_{30}=0$. We have $b_{20}\not=0$. By multiplying $b_{20}^{-1}$, we may assume that 
\begin{eqnarray*}
\delta &=& (a_{12}xy^2+a_{20}x^2 + a_{02}y^2 + a_{10}x ) \partial_x  + (a_{12}y^3+x^2 + b_{02}y^2 + a_{10}y) \partial_y \\
 &=& \dfrac{1}{z}((a_{02} zw^2+a_{10} z^2+a_{12}w^2 +a_{20}z ) \partial_z + ( a_{02}w^3+b_{02}w^2  + a_{20}w +1) \partial_w)\\
 &=& \dfrac{1}{u}((uv^2+a_{10} u^2 + b_{02}u+a_{12} ) \partial_u +( v^3+ a_{20}v^2 + b_{02}v + a_{02} ) \partial_v).
\end{eqnarray*}
Set $G=0$. Then we have $x^2=a_{12}y^3+ b_{02}y^2 + a_{10}y$ and can check that $F^2= y(a_{12}^3y^6+a_{12}^2 (b_{02}+a_{20}^2)y^5+a_{12}^2 a_{10}y^4
+(a_{20}b_{02}+a_{02})^2y^3+a_{10}^2 a_{12}y^2+a_{10}^2(a_{20}^2+b_{02})y+a_{10}^3)$. Let $f=a_{12}^3y^6+a_{12}^2(b_{02}+a_{20}^2)y^5+a_{12}^2 a_{10}y^4
+(a_{20}b_{02}+a_{02})^2y^3+a_{10}^2 a_{12}y^2+a_{10}^2(a_{20}^2+b_{02})y+a_{10}^3$. If two or all out of $a_{12}$, $a_{20}b_{02}+a_{02}$ 
and $a_{10}$ are non-zero, then the equation $f=0$ has a non-zero root, since two or all terms out of $a_{12}^3y^6$, $(a_{20}b_{02}+a_{02})^2y^3$ 
and $a_{10}^3$ are non-zero. Then there exist two or more singular points, which is the case (A) or (B). Thus we may assume that two or all out 
of $a_{12}$, $a_{20}b_{02}+a_{02}$ and $a_{10}$ are zero. We consider the four cases such that two or all out of $a_{12}$, $a_{20}b_{02}+a_{02}$ and 
$a_{10}$ are zero. The following table summarizes what we know. 
$$\begin{array}{c|c|c|c|c|c} 
 & a_{12}  & a_{20}b_{02}+a_{02} & a_{10} & \deg L & \\ \hline
(1) & k^{\ast} & 0 & 0 & -1 & E_7^0 \\ \hline
(2) & 0 & k^{\ast} & 0 & & \mbox{(A) or (B)} \\ \hline
(3) & 0 & 0 & k^{\ast} & & \mbox{(A) or (B)} \\ \hline
(\clubsuit) & 0 & 0 & 0 & & 
\end{array}$$
$(\clubsuit)$ Since $a_{02}=a_{20}b_{02}$, we have $\delta=a_{20}(x^2 + b_{02}y^2) \partial_x  + (x^2 + b_{02}y^2) \partial_y$. This contradicts to the 
assumption (iii). \\
(1) Suppose that $b_{02}+a_{20}^2\not=0$. Since two terms $a_{12}^3y^6$ and $a_{12}^2(b_{02}+a_{20}^2)y^5$ are non-zero, the equation $f=0$ has 
a non-zero root. Then there exist two or more singular points, which is the case (A) or (B). Thus we may assume that $b_{02}+a_{20}^2=0$. We have 
$b_{02}=a_{20}^2$, $a_{02}=a_{20}b_{02}=a_{20}^3$ and 
\begin{eqnarray*}
\delta &=& (a_{12}xy^2+a_{20}x^2 + a_{20}^3y^2) \partial_x  + (a_{12}y^3+x^2 + a_{20}^2y^2) \partial_y \\
 &=& \dfrac{1}{z}((a_{20}^3 zw^2+a_{12}w^2 +a_{20}z ) \partial_z + ( a_{20}^3w^3+a_{20}^2w^2  + a_{20}w +1) \partial_w)\\
 &=& \dfrac{1}{u}((uv^2+ a_{20}^2u+a_{12} ) \partial_u +( v^3+ a_{20}v^2 + a_{20}^2v + a_{20}^3 ) \partial_v).
\end{eqnarray*}
We see that $\deg L=-1$. On $U_0$, we have $k[x,y]^{\delta}=k[x^2,y^2,a_{12}xy^3+a_{20}x^2y + a_{20}^3y^3+x^3+a_{20}^2xy^2]
=k[x^2,y^2,(x+a_{20}y)^3+a_{12}xy^3]=k[x^2,y^2,(x+a_{20}y)^3+a_{12}(x+a_{20}y)y^3+a_{12}a_{20}y^4]=k[x^2,y^2,(x+a_{20}y)^3+a_{12}(x+a_{20}y)y^3]
\cong k[x^2,y^2,x^3+a_{12}xy^3]\cong k[x^2,y^2,x^3+xy^3]\cong k[X,Y,Z]/(Z^2+X^3+XY^3)$. On $U_1$, we see that $w=a_{20}^{-1}$ is the only one root 
of the equation $a_{20}^3w^3+a_{20}^2w^2  + a_{20}w +1=0$, since $a_{20}^3w^3+a_{20}^2w^2  + a_{20}w +1=(a_{20}w+1)^3$. Thus $Y$ is smooth on 
$\pi(U_1)$, since $(a_{20}^3 zw^2+a_{12}w^2 +a_{20}z)|_{w=a_{20}^{-1}}=a_{12}a_{20}^{-2}\not=0$. We see also that $Y$ is smooth on $\pi(U_2)$ in the same 
way. Therefore the configuration of the singular points of $Y$ is the type $E_7^0$. \\
(2) We have $\delta=z^{-1}(z(a_{02} w^2+a_{20})\partial_z+(a_{02}w^3+b_{02}w^2+a_{20}w+1)\partial_w)$. Let $g=a_{02}w^3+b_{02}w^2+a_{20}w+1$. Since 
$a_{20}b_{02}+a_{02}\not=0$, we see that $a_{20}$, $b_{02}$ or $a_{02}$ is non-zero. Thus the equation $g=0$ has a non-zero root. Then there exist two 
or more singular points, which is the case (A) or (B). \\
(3) As with (1), $b_{02}=a_{20}^2$ and $a_{02}=a_{20}^3$. Thus $\delta=z^{-1}(z(a_{20}^3w^2+a_{10} z+a_{20} ) \partial_z + (a_{20}w+1)^3 \partial_w)$. 
We see that $Y$ has a singular point at the image of $(0,a_{20}^{-1})$ on $U_1$, which is the case (A) or (B).\\
(C-2) Suppose that $a_{30}\not=0$. By multiplying $a_{30}^{-1}$, we may assume that
\begin{eqnarray*}
\delta &=& (x^3+a_{12}xy^2+a_{20}x^2 + a_{02}y^2 + a_{10}x ) \partial_x  + (x^2y+a_{12}y^3+b_{20}x^2 + b_{02}y^2 + a_{10}y) \partial_y \\
&=&\dfrac{1}{z}((a_{10} z^2+(a_{02}w^2+a_{20})z+(a_{12}w^2+1)) \partial_z + ( a_{02}w^3+b_{02}w^2+a_{20}w+b_{20}) \partial_w\\
&=&\dfrac{1}{u}((a_{10} u^2+(b_{20} v^2 + b_{02})u+(v^2+a_{12})) \partial_u +( b_{20}v^3+ a_{20}v^2 + b_{02}v + a_{02})) \partial_v.
\end{eqnarray*} 
(C-2-1) Suppose that $a_{10}\not=0$. If $a_{02}$, $b_{02}$ or $a_{20}$ is non-zero, then there exists a root $w=\alpha$ of the equation 
$a_{02}w^3+b_{02}w^2+a_{20}w+b_{20}=0$. Let $z=\beta$ be a root of the equation $a_{10} z^2+(a_{02}\alpha^2+a_{20})z+a_{12}\alpha^2+1=0$. Then 
$Y$ has a singular point at the image of $(\beta, \alpha)$ on $U_1$, which is the case (A) or (B). If $a_{02}=b_{02}=a_{20}=0$, then $a_{12}\not=0$ 
and $\delta = (x^3+a_{12}xy^2+ a_{10}x ) \partial_x  + (x^2y+a_{12}y^3+b_{20}x^2 + a_{10}y) \partial_y$. We see that $Y$ has a singular point at the 
image of $(0,\sqrt{a_{12}^{-1}a_{10}})$ on $U_0$, which is the case (A) or (B). \\
(C-2-2) Suppose that $a_{10}=b_{20}=0$. We have 
\begin{eqnarray*}
\delta &=& (x^3+a_{12}xy^2+a_{20}x^2 + a_{02}y^2) \partial_x  + (x^2y+a_{12}y^3+ b_{02}y^2) \partial_y \\
 &=& \dfrac{1}{z}((a_{02}zw^2+a_{12}w^2+a_{20}z+1) \partial_z + ( a_{02}w^3+b_{02}w^2+a_{20}w) \partial_w) \\
 &=& \dfrac{1}{u}((b_{02}u+v^2+a_{12}) \partial_u +( a_{20}v^2 + b_{02}v + a_{02}) \partial_v).
\end{eqnarray*}
Assume that $a_{20}\not=0$. If $F$ and $G$ have no common factors, then $Y$ has a singular point at the image of $(a_{20},0)$ on $U_0$, which 
is the case (A) or (B). Thus we may assume that $a_{20}=0$. We consider the eight cases such that each coefficient $a_{12}$, $a_{02}$ and $b_{02}$ 
is zero or not. The following table summarizes what we know. 
$$\begin{array}{c|c|c|c|c|c}
 & a_{12} & a_{02} & b_{02} & \deg L & \\ \hline
(\clubsuit) & 0 & 0 & 0 & & \delta=x^3\partial_x  + x^2y \partial_y \\ \hline
(1) & 0 & 0 & k^{\ast} & & \mbox{(A) or (B)} \\ \hline
(2) & 0 & k^{\ast} & 0 & -1 & E_7^0 \\ \hline
(3) & 0 & k^{\ast} & k^{\ast} & & \mbox{(A) or (B)} \\ \hline
(\clubsuit) & k^{\ast} & 0 & 0 & & \delta=x(x^2+a_{12}y^2) \partial_x  + y(x^2+a_{12}y^2) \partial_y \\ \hline
(4) & k^{\ast} & 0 & k^{\ast} & & \mbox{(A) or (B)} \\ \hline
(5) & k^{\ast} & k^{\ast} & 0 & -1 & E_7^0 \\ \hline
(6) & k^{\ast} & k^{\ast} & k^{\ast} & & \mbox{(A) or (B)}
\end{array}$$
The cases ($\clubsuit$) contradict to the assumption (iii). \\
(1) We have $\delta=u^{-1}((b_{02}u+v^2) \partial_u + b_{02}v \partial_v)$. Thus $Y$ has a singular point at the image of the origin of $U_2$, which 
is the case (A) or (B). \\
(2) We have 
$$\delta = (x^3+ a_{02}y^2) \partial_x  + x^2y\partial_y 
 = \dfrac{1}{z}((a_{02}zw^2+1) \partial_z + a_{02}w^3\partial_w) 
 = \dfrac{1}{u}(v^2 \partial_u + a_{02} \partial_v).$$
We see that $\deg L=-1$ and $Y$ is smooth on $\pi(U_1)$ and $\pi(U_2)$. On $U_0$, we have $k[x,y]^{\delta}=k[x^2,y^2, x^3y+a_{02}y^3]
\cong k[x^2,y^2, x^3y+y^3]\cong k[X,Y,Z]/(Z^2+X^3Y+Y^3)= k[X,Y,Z]/(Z^2+X^3+XY^3)$. Thus the configuration of the singular points of $Y$ is the 
type $E_7^0$. \\
(3) $\delta =(x^3 + a_{02}y^2) \partial_x  + y(x^2+ b_{02}y) \partial_y$: We see that $Y$ has a singular point at the image of 
$(b_{02}^2a_{02}^{-1}, b_{02}^3a_{02}^{-2})$ on $U_0$, which is the case (A) or (B). \\
(4) $\delta=x(x^2+a_{12}y^2) \partial_x  + y(x^2+a_{12}y^2+ b_{02}y) \partial_y$: We see that $Y$ has a singular point at the image of 
$(0,b_{02}a_{12}^{-1})$ on $U_0$, which is the case (A) or (B). \\
(5) We have 
\begin{eqnarray*}
\delta &=& (x^3+a_{12}xy^2 + a_{02}y^2) \partial_x + (x^2y+a_{12}y^3) \partial_y \\ 
 &=& \dfrac{1}{z}((a_{02}z w^2+a_{12}w^2+1 ) \partial_z+ a_{02} w^3 \partial_w) = \dfrac{1}{u}((v^2+a_{12}) \partial_u +a_{02}\partial_v).
\end{eqnarray*}
We see that $\deg L=-1$ and $Y$ is smooth on $\pi(U_1)$ and $\pi(U_2)$. On $U_0$, we have $k[x,y]^{\delta}=k[x^2,y^2, x^3y+a_{12}xy^3+a_{02}y^3]
\cong k[X,Y,Z]/(Z^2+X^3Y+a_{12}^2XY^3+a_{02}^2Y^3)= k[X,Y,Z]/(Z^2+X^3Y+(a_{12}^2X+a_{02}^2)Y^3)$. Let $R=k[X,Y,Z]/(Z^2+X^3Y+(a_{12}^2X+a_{02}^2)Y^3)$ 
and $\mathfrak{m}=(X,Y,Z)R$. The $\mathfrak{m}$-adic completion of $R$ is isomorphic to $k[[X,Y,Z]]/(Z^2+X^3Y+(a_{12}^2X+a_{02}^2)Y^3)
\cong k[[X,Y,Z]]/(Z^2+X^3Y+Y^3)\cong k[[X,Y,Z]]/(Z^2+X^3+XY^3)$. Thus the configuration of the singular points of $Y$ is the type $E_7^0$. \\
(6) $\delta=(x^3+a_{12}xy^2+ a_{02}y^2) \partial_x  + y(x^2+a_{12}y^2+ b_{02}y) \partial_y$: Set $x^2+a_{12}y^2+ b_{02}y=0$. Then we have 
$x^2=y(a_{12}y+ b_{02})$ and can check that $F^2=y^3((a_{12}b_{02}^2+a_{02}^2)y+b_{02}^3)$. \\
(6-1) Suppose that $a_{12}b_{02}^2+a_{02}^2\not=0$. Let $\alpha=b_{02}^3(a_{12}b_{02}^2+a_{02}^2)^{-1}$ and $\beta=\sqrt{\alpha(a_{12}\alpha+ b_{02})}$. 
We see that $(F(\beta,\alpha))^2=0$, so that $F(\beta,\alpha)=G(\beta,\alpha)=0$. Thus $Y$ has a singular point at the image of $(\beta,\alpha)$ on 
$U_0$, which is the case (A) or (B). \\
(6-2) Suppose that $a_{12}b_{02}^2+a_{02}^2=0$. Then we have $a_{12}=a_{02}^2b_{02}^{-2}$ and $\delta=z^{-1}((a_{02}zw^2+a_{02}^2b_{02}^{-2}w^2+1)\partial_z
+w^2 ( a_{02}w+b_{02}) \partial_w)$. We see that $Y$ has a singular point at the image of $(0,b_{02}a_{02}^{-1})$ on $U_1$, which is the case (A) or (B). \\
(C-2-3) Suppose that $a_{10}=0$ and $b_{20}\not=0$. We have 
\begin{eqnarray*}
\delta &=& (x^3+a_{12}xy^2+a_{20}x^2 + a_{02}y^2) \partial_x  + (x^2y+a_{12}y^3+b_{20}x^2 + b_{02}y^2 ) \partial_y \\
 &=&\dfrac{1}{z}((a_{02}zw^2+a_{12}w^2+a_{20}z+1) \partial_z + ( a_{02}w^3+b_{02}w^2+a_{20}w+b_{20}) \partial_w)\\
 &=& \dfrac{1}{u}((b_{20} uv^2 +v^2+ b_{02}u+a_{12}) \partial_u +( b_{20}v^3+ a_{20}v^2 + b_{02}v + a_{02}) \partial_v).
\end{eqnarray*}
We consider the sixteen cases such that each coefficient $a_{12}$, $a_{20}$, $a_{02}$ and $b_{02}$ is zero or not. The following table summarizes what we know. 
$$\begin{array}{c|c|c|c|c|c|c}
 & a_{12} & a_{20} & a_{02} & b_{02} & \deg L & \\ \hline
(\diamondsuit) & 0 & 0 & 0 & 0 & & \\ \hline
(\diamondsuit) & 0 & 0 & 0 & k^{\ast} & & \\ \hline
(1) & 0 & 0 & k^{\ast} & 0 & & \mbox{(A) or (B)} \\ \hline
(2) & 0 & 0 & k^{\ast} & k^{\ast} & & \mbox{(A) or (B)} \\ \hline
(\diamondsuit) & 0 & k^{\ast} & 0 & 0 & & \\ \hline
(\diamondsuit) & 0 & k^{\ast} & 0 & k^{\ast} & & \\ \hline
(3) & 0 & k^{\ast} & k^{\ast} & 0 & & \mbox{(A) or (B)} \\ \hline
(4) & 0 & k^{\ast} & k^{\ast} & k^{\ast} & & \mbox{(A), (B) or $E_7^0$} \\ \hline
(5) & k^{\ast} & 0 & 0 & 0 & -1 & E_7^0\\ \hline
(6) & k^{\ast} & 0 & 0 & k^{\ast} & & \mbox{(A) or (B)} \\ \hline
(7) & k^{\ast} & 0 & k^{\ast} & 0 & & \mbox{(A) or (B)} \\ \hline
(8) & k^{\ast} & 0 & k^{\ast} & k^{\ast} & & \mbox{(A) or (B)} \\ \hline
(9) & k^{\ast} & k^{\ast} & 0 & 0 & & \mbox{(A) or (B)} \\ \hline
(10) & k^{\ast} & k^{\ast} & 0 & k^{\ast} & & \mbox{(A) or (B)} \\ \hline
(11) & k^{\ast} & k^{\ast} & k^{\ast} & 0 & & \mbox{(A) or (B)} \\ \hline
(12) & k^{\ast} & k^{\ast} & k^{\ast} & k^{\ast} & & \mbox{(A), (B) or $E_7^0$} 
\end{array}$$
In the case (4) or (12), we have $\deg L=-1$ if the configuration of the singular points of $Y$ is the type $E_7^0$. \\
$(\diamondsuit)$ If $F$ and $G$ have no common factors, then we see that $Y$ has a singular point at the origin of $U_2$, since $a_{12}=a_{02}=0$. 
This is the case (A) or (B). \\
(1) $\delta=(x^3+ a_{02}y^2) \partial_x  + x^2(y+b_{20}) \partial_y$: Let $x=\sqrt[3]{a_{02}b_{20}^2}$ be a root of the equation $x^3+a_{02}b_{20}^2=0$. 
We see that $Y$ has a singular point at the image of $(\sqrt[3]{a_{02}b_{20}^2},b_{20})$ on $U_0$, which is the case (A) or (B). \\
(2) We have $\delta=z^{-1}((a_{02}zw^2+1) \partial_z + ( a_{02}w^3+b_{02}w^2+b_{20}) \partial_w)$. Since $(a_{02}w^3+b_{02}w^2+b_{20})|_{w=0}=b_{20}\not=0$, 
the equation $a_{02}w^3+b_{02}w^2+b_{20}=0$ has a non-zero root $w=\alpha$. We see that $Y$ has a singular point at the image of 
$(a_{02}^{-1}\alpha^{-2}, \alpha)$ on $U_1$, which is the case (A) or (B). \\
(3) $\delta=(x^3+a_{20}x^2 + a_{02}y^2) \partial_x  + x^2(y+b_{20}) \partial_y$: Since $(x^3+a_{20}x^2 + a_{02}b_{20}^2)|_{x=0}=a_{02}b_{20}^2\not=0$, 
the equation $x^3+a_{20}x^2 + a_{02}b_{20}^2=0$ has a non-zero root $x=\alpha$. We see that $Y$ has a singular point at the image of $(\alpha, b_{20})$ 
on $U_0$, which is the case (A) or (B). \\
(4) We have 
\begin{eqnarray*}
\delta &=& (x^3+a_{20}x^2 + a_{02}y^2) \partial_x  + (x^2y+b_{20}x^2 + b_{02}y^2 ) \partial_y \\
 &=&\dfrac{1}{z}(((a_{02}w^2+a_{20})z+a_{12}w^2+1) \partial_z + ( a_{02}w^3+b_{02}w^2+a_{20}w+b_{20}) \partial_w)\\
 &=& \dfrac{1}{u}((b_{20} uv^2 +v^2+ b_{02}u) \partial_u +( b_{20}v^3+ a_{20}v^2 + b_{02}v + a_{02}) \partial_v).
\end{eqnarray*}
If $F$ and $G$ have no common factors, then $\deg L=-1$. Let $h(w)=w^3+b_{02}a_{02}^{-1}w^2+a_{20}a_{02}^{-1}w+b_{20}a_{02}^{-1}$. \\
(4-1) Suppose that the equation $h=0$ has a non-zero cubic root $w=\alpha$. Then $\alpha=b_{02}a_{02}^{-1}$, $\alpha^2=a_{20}a_{02}^{-1}$ 
and $\alpha^3=b_{20}a_{02}^{-1}$, so that $b_{02}=\alpha a_{02}$, $a_{20}=\alpha^2 a_{02}$ and $b_{20}=\alpha^3a_{02}$. On $U_0$, we have $k[x,y]^{\delta}
=k[x^2,y^2,x^3y+a_{20}x^2y+a_{02}y^3+b_{20}x^3+b_{02}xy^2]=k[x^2,y^2,x^3y+\alpha^2a_{02}x^2y+a_{02}y^3+\alpha^3a_{02}x^3+\alpha a_{02}xy^2]
=k[x^2,y^2,x^3y+a_{02}(\alpha x+y)^3]=k[x^2,y^2, x^3(\alpha x+y)+\alpha x^4+a_{02}(\alpha x+y)^3]=k[x^2,y^2, x^3(\alpha x+y)+a_{02}(\alpha x+y)^3]
\cong k[x^2,y^2, x^3y+y^3]\cong k[X,Y,Z]/(Z^2+X^3Y+Y^3)\cong k[X,Y,Z]/(Z^2+X^3+XY^3)$. We see that $Y$ has a singular point, which is an 
$E_7^0$-singularity, at the image of the origin of $U_0$ and is smooth except this singular point on $\pi(U_0)$. Since 
$(a_{02}zw^2+a_{20}z+1)_{w=\alpha}=(a_{02}zw^2+\alpha^2 a_{02}z+1)_{w=\alpha}=1\not=0$, we see that $Y$ is smooth on $\pi(U_1)$. Since 
$(b_{20}v^3+ a_{20}v^2 + b_{02}v + a_{02})|_{u=v=0}=a_{02}\not=0$, we see also that $Y$ is smooth at the image of the origin of $U_2$. Thus the 
configuration of the singular points of $Y$ is the type $E_7^0$. \\
(4-2) Suppose that the equation $h=0$ has different roots. Then there exists an element $\alpha$ of $k$ such that $h(\alpha)=0$ and 
$a_{02}\alpha^2+a_{20}\not=0$, since $a_{02}\beta^2+a_{20}=a_{02}\gamma^2+a_{20}=0$ implies $\beta=\gamma$ for $\beta, \gamma\in k$. We see that $Y$ has a singular 
point at the image of $((a_{12}\alpha^2+1)(a_{02}\alpha^2+a_{20})^{-1},\alpha)$ on $U_1$, which is the case (A) or (B). \\
(5) We have 
\begin{eqnarray*}
\delta &=& (x^3+a_{12}xy^2 ) \partial_x + (x^2y+a_{12}y^3+b_{20}x^2 ) \partial_y \\
 &=& \dfrac{1}{z}((a_{12}w^2+1) \partial_z + b_{20} \partial_w) = \dfrac{1}{u}((b_{20}uv^2+v^2+a_{12}) \partial_u + b_{20}v^3 \partial_v).
\end{eqnarray*}
We see that $\deg L=-1$ and $Y$ is smooth on $\pi(U_1)$ and $\pi(U_2)$. 
On $U_0$, we have $k[x^2,y^2, x^3y+a_{12}xy^3+b_{20}x^3]\cong k[X,Y,Z]/(Z^2+X^3Y+a_{12}^2XY^3+b_{20}^2X^3)\cong k[X,Y,Z]/(Z^2+(Y+b_{02}^2)X^3+a_{12}^2XY^3)$. 
Let $R=k[X,Y,Z]/(Z^2+(Y+b_{02}^2)X^3+a_{12}^2XY^3)$ and $\mathfrak{m}=(X,Y,Z)R$. The $\mathfrak{m}$-adic completion of $R$ is isomorphic to 
$k[[X,Y,Z]]/(Z^2+(Y+b_{02}^2)X^3+a_{12}^2XY^3)\cong k[[X,Y,Z]]/(Z^2+X^3+XY^3)$. Thus the configuration of the singular points of $Y$ is the type $E_7^0$. \\
(6) $\delta=(x^3+a_{12}xy^2) \partial_x  + (x^2y+a_{12}y^3+b_{20}x^2 + b_{02}y^2 ) \partial_y$: We see that $Y$ has a singular point at the image of 
$(0,a_{12}^{-1}b_{20})$ on $U_0$, which is the case (A) or (B). \\
(7) $\delta=(x^3+a_{12}xy^2 + a_{02}y^2) \partial_x  + ((y+b_{20})x^2+a_{12}y^3) \partial_y$: Set $G=0$. Then we have $x^2=a_{12}y^3(y+b_{20})^{-1}$ and 
can check that $(y+b_{20})^3F^2=y^4((a_{12}^3b_{20}^2+a_{02}^2)y^3+a_{02}^2b_{20}y^2+a_{02}^2b_{20}^2y+a_{02}^2b_{20}^3)$. Let 
$f(y)=(a_{12}^3b_{20}^2+a_{02}^2)y^3+a_{02}^2b_{20}y^2+a_{02}^2b_{20}^2y+a_{02}^2b_{20}^3$. Since $f(0)=a_{02}^2b_{20}^3\not=0$ and 
$f(b_{20})=a_{12}^3b_{20}^5\not=0$, the equation $f=0$ has a non-zero root $y=\alpha$ such that $\alpha\not=b_{20}$. Let 
$\beta=\sqrt{a_{12}\alpha^3(\alpha+b_{20})^{-1}}$. We see that $(\alpha+b_{02})^{3}(F(\beta , \alpha))^2=0$, so that $F(\beta , \alpha)=G(\beta , \alpha)=0$. 
Thus $Y$ has a singular point at the image of $(\beta, \alpha)$ on $U_0$, which is the case (A) or (B). \\
(8) $\delta=(x^3+a_{12}xy^2+ a_{02}y^2) \partial_x  + ((y+b_{20})x^2+(a_{12}y + b_{02})y^2 ) \partial_y$: Set $G=0$. Then we have 
$x^2=(a_{12}y+ b_{02})y^2(y+b_{20})^{-1}$ and can check that 
$(y+b_{20})^3F^2=y^4((a_{12}b_{02}^2+a_{12}^3b_{20}^2+a_{02}^2)y^3+(b_{02}^3+a_{12}^2b_{02}b_{20}^2+a_{02}^2b_{20})y^2+a_{02}^2b_{20}^2y^5+a_{02}^2b_{20}^3)$. 
Let $f(y)=(a_{12}b_{02}^2+a_{12}^3b_{20}^2+a_{02}^2)y^3+(b_{02}^3+a_{12}^2b_{02}b_{20}^2+a_{02}^2b_{20})y^2+a_{02}^2b_{20}^2y+a_{02}^2b_{20}^3$. 
Then we can verify that $f(b_{20})=b_{20}^2(a_{12}b_{20}+b_{02})^3$. \\
(8-1) Suppose that $a_{12}b_{20}+b_{02}\not=0$. Since $f(0)=a_{02}^2b_{20}^3\not=0$ and $f(b_{20})=b_{20}^2(a_{12}b_{20}+b_{02})^3\not=0$, the equation 
$f=0$ has a non-zero root $y=\alpha$ such that $\alpha\not=b_{20}$. Let $\beta=\sqrt{(a_{12}\alpha+ b_{02})\alpha^2(\alpha+b_{20})^{-1}}$. We see that 
$(\alpha+b_{20})^3(F(\beta, \alpha))^2=0$, so that $F(\beta, \alpha)=G(\beta, \alpha)=0$. Thus $Y$ has a singular point at the image of $(\beta, \alpha)$ 
on $U_0$, which is the case (A) or (B). \\
(8-2) Suppose that $a_{12}b_{20}+b_{02}=0$. Since $(x^3+a_{12}b_{20}^2x+ a_{02}b_{20}^2)|_{x=0}=a_{02}b_{20}^2\not=0$, the equation 
$F(x,b_{20})=x^3+a_{12}b_{20}^2x+ a_{02}b_{20}^2=0$ has a non-zero root $x=\alpha$. Since $a_{12}b_{20}+b_{02}=0$, we have $G(\alpha, b_{20})=0$. Thus $Y$ has 
a singular point at the image of $(\alpha, b_{20})$ on $U_0$, which is the case (A) or (B). \\
(9) We have $\delta=z^{-1}((a_{20}z+(a_{12}w^2+1)) \partial_z + (a_{20}w+b_{20}) \partial_w)$. Thus $Y$ has a singular point at the image of 
$(a_{20}^{-1}(a_{12}a_{20}^{-2}b_{20}^2+1), a_{20}^{-1}b_{20})$ on $U_1$, which is the case (A) or (B). \\
(10) We have $\delta=z^{-1}((a_{20}z+(a_{12}w^2+1)) \partial_z + (b_{02}w^2+a_{20}w+b_{20}) \partial_w)$. Since $(b_{02}w^2+a_{20}w+b_{20})|_{w=0}=b_{20}\not=0$, 
the equation $b_{02}w^2+a_{20}w+b_{20}=0$ has a non-zero root $w=\alpha$. We see that $Y$ has a singular point at the image of 
$(a_{20}^{-1}(a_{12}\alpha^2+1),\alpha)$ on $U_1$, which is the case (A) or (B). \\
(11) We have $\delta=z^{-1}(((a_{02}w^2+a_{20})z+(a_{12}w^2+1)) \partial_z + ( a_{02}w^3+a_{20}w+b_{20}) \partial_w)$. Since 
$(a_{02}w^3+a_{20}w+b_{20})|_{w=0}=b_{20}\not=0$, the equation $a_{02}w^3+a_{20}w+b_{20}=0$ has a non-zero root $w=\alpha$. Note that 
$a_{02}\alpha^2+a_{20}=b_{20}\alpha^{-1}\not=0$, since $a_{02}\alpha^3+a_{20}\alpha+b_{20}=0$. We see that $Y$ has a singular point at the image of 
$((a_{02}\alpha^2+a_{20})^{-1}(a_{12}\alpha^2+1),\alpha)$ on $U_1$, which is the case (A) or (B). \\
(12) We have 
\begin{eqnarray*}
\delta &=& (x^3+a_{12}xy^2+a_{20}x^2 + a_{02}y^2) \partial_x  + (x^2y+a_{12}y^3+b_{20}x^2 + b_{02}y^2 ) \partial_y \\
 &=&\dfrac{1}{z}(((a_{02}w^2+a_{20})z+a_{12}w^2+1) \partial_z + ( a_{02}w^3+b_{02}w^2+a_{20}w+b_{20}) \partial_w)\\
 &=& \dfrac{1}{u}((b_{20} uv^2 +v^2+ b_{02}u+a_{12}) \partial_u +( b_{20}v^3+ a_{20}v^2 + b_{02}v + a_{02}) \partial_v).
\end{eqnarray*}
If $F$ and $G$ have no common factors, then $\deg L=-1$. Let $h(w)=w^3+b_{02}a_{02}^{-1}w^2+a_{20}a_{02}^{-1}w+b_{20}a_{02}^{-1}$. \\
(12-1) Suppose that the equation $h=0$ has a non-zero cubic root $w=\alpha$. Then $\alpha=b_{02}a_{02}^{-1}$, $\alpha^2=a_{20}a_{02}^{-1}$ 
and $\alpha^3=b_{20}a_{02}^{-1}$, so that $b_{02}=\alpha a_{02}$, $a_{20}=\alpha^2 a_{02}$ and $b_{20}=\alpha^3a_{02}$. If $1+a_{12}\alpha^2=0$, then 
\begin{eqnarray*}
\delta &=& \dfrac{1}{z}(((a_{02}w^2+\alpha^2 a_{02})z+a_{12}w^2+a_{12}\alpha^2) \partial_z + a_{02}(w+\alpha)^3 \partial_w) \\
&=& \dfrac{(w+\alpha)^2}{z}((a_{02}z+a_{12}) \partial_z + a_{02}(w+\alpha) \partial_w).
\end{eqnarray*}
We see that $Y$ has a singular point at the image of $(a_{02}^{-1}a_{12},\alpha)$ on $U_1$, which is the case (A) or (B). Thus we may assume that 
$1+a_{12}\alpha^2\not=0$. On $U_0$, we have 
\begin{eqnarray*}
\lefteqn{k[x,y]^{\delta}=k[x^2,y^2,x^3y+a_{12}xy^3+a_{20}x^2y+a_{02}y^3+b_{20}x^3+b_{02}xy^2]}\\
&=& k[x^2,y^2,x^3y+a_{12}xy^3+\alpha^2a_{02}x^2y+a_{02}y^3+\alpha^3a_{02}x^3+\alpha a_{02}xy^2] \\
&=& k[x^2,y^2,x^3(\alpha x+y)+\alpha x^4+a_{12}x(\alpha x+y)^3+a_{12}x(\alpha^3x^3+\alpha^2x^2y+\alpha xy^2)+a_{02}(\alpha x+y)^3]\\
&=& k[x^2,y^2,x^3(\alpha x+y)+(a_{12}x+a_{02})(\alpha x+y)^3+a_{12}\alpha^2x^3y]\\
&=& k[x^2,y^2,x^3(\alpha x+y)+(a_{12}x+a_{02})(\alpha x+y)^3+a_{12}\alpha^2x^3(\alpha x+y)+a_{12}\alpha^3x^4]\\
&=& k[x^2,y^2,(1+a_{12}\alpha^2)x^3(\alpha x+y)+(a_{12}x+a_{02})(\alpha x+y)^3]\\
&\cong& k[x^2,y^2,(1+a_{12}\alpha^2)x^3y+(a_{12}x+a_{02})y^3]\\
&\cong&k[X,Y,Z]/(Z^2+(1+a_{12}\alpha^2)^2X^3Y+(a_{12}^2X+a_{02}^2)Y^3).
\end{eqnarray*}
Let $R=k[X,Y,Z]/(Z^2+(1+a_{12}^2\alpha^4)X^3Y+(a_{12}^2X+a_{02}^2)Y^3)$ and $\mathfrak{m}=(X,Y,Z)R$. The $\mathfrak{m}$-adic completion of 
$R$ is isomorphic to $k[[X,Y,Z]]/(Z^2+(1+a_{12}^2\alpha^4)X^3Y+(a_{12}^2X+a_{02}^2)Y^3)\cong k[[X,Y,Z]]/(Z^2+X^3Y+Y^3)\cong k[[X,Y,Z]]/(Z^2+X^3+XY^3)$. 
Since $((a_{02}w^2+a_{20})z+a_{12}w^2+1)_{w=\alpha}=((a_{02}w^2+\alpha^2 a_{02})z+a_{12}w^2+1)_{w=\alpha}=a_{12}\alpha^2+1\not=0$, we see that $Y$ 
is smooth on $\pi(U_1)$. Since $(b_{20}v^3+ a_{20}v^2 + b_{02}v + a_{02})|_{v=0}=a_{02}\not=0$, we see also that $Y$ is smooth on the image of the line 
defined by $v=0$ on $U_2$. Thus the configuration of the singular points of $Y$ is the type $E_7^0$. \\
(12-2) Suppose that the equation $h=0$ has different roots. Then there exists an element $\alpha$ of $k$ such that $h(\alpha)=0$ and $a_{02}\alpha^2+a_{20}\not=0$, 
since $a_{02}\beta^2+a_{20}=a_{02}\gamma^2+a_{20}=0$ implies $\beta=\gamma$ for $\beta, \gamma \in k$. We see that $Y$ has a singular point at the image of 
$((a_{12}\alpha^2+1)(a_{02}\alpha^2+a_{20})^{-1},\alpha)$ on $U_1$, which is the case (A) or (B). 
\end{proof}

Below Lemma~\ref{1}, we asked if the case that $Y$ has $m$ singular points exists for every $m$ such that $1\leq m\leq n^2-3n+3$, 
where $n=\deg L$. By Theorem~\ref{main}, we see that an invertible $1$-foliation $L$ of degree $-1$ exists and the situation that $Y$ 
has $m$ singular points happens only for $m=1, 2, 4$ and $7$ if $n=-1$. 

\bibliographystyle{amsplain}

\begin{thebibliography}{1}

\bibitem{AA}
Annetta~G. Aramova and Luchezar~L. Avramov, \emph{Singularities of quotients by
  vector fields in characteristic {$p$}}, Math. Ann. \textbf{273} (1986),
  no.~4, 629--645. \MR{826462}

\bibitem{A}
M.~Artin, \emph{Coverings of the rational double points in characteristic
  {$p$}}, Complex analysis and algebraic geometry, Iwanami Shoten, Tokyo, 1977,
  pp.~11--22. \MR{0450263}

\bibitem{E}
Torsten Ekedahl, \emph{Canonical models of surfaces of general type in positive
  characteristic}, Inst. Hautes \'Etudes Sci. Publ. Math. (1988), no.~67,
  97--144. \MR{972344}

\bibitem{GR}
Richard Ganong and Peter Russell, \emph{Derivations with only divisorial
  singularities on rational and ruled surfaces}, J. Pure Appl. Algebra
  \textbf{26} (1982), no.~2, 165--182. \MR{675013}

\bibitem{H}
Robin Hartshorne, \emph{Algebraic geometry}, Springer-Verlag, New
  York-Heidelberg, 1977, Graduate Texts in Mathematics, No. 52. \MR{0463157}

\bibitem{Hirokado}
Masayuki Hirokado, \emph{Zariski surfaces as quotients of {H}irzebruch surfaces
  by 1-foliations}, Yokohama Math. J. \textbf{47} (2000), no.~2, 103--120.
  \MR{1763776}

\bibitem{J}
Nathan Jacobson, \emph{Basic algebra. {II}}, second ed., W. H. Freeman and
  Company, New York, 1989. \MR{1009787}

\bibitem{RS}
A.~N. Rudakov and I.~R. {\v{S}}afarevi{\v{c}}, \emph{Inseparable morphisms of
  algebraic surfaces}, Izv. Akad. Nauk SSSR Ser. Mat. \textbf{40} (1976),
  no.~6, 1269--1307, 1439. \MR{0460344}

\bibitem{S4}
Tadakazu Sawada, \emph{Classification of globally {F}-regular {$F$}-sandwiches
  of {H}irzebruch surfaces}, preprint, arXiv:1604.00745 (2016).

\end{thebibliography}
\providecommand{\bysame}{\leavevmode\hbox to3em{\hrulefill}\thinspace}
\providecommand{\MR}{\relax\ifhmode\unskip\space\fi MR }
\providecommand{\MRhref}[2]{%
  \href{http://www.ams.org/mathscinet-getitem?mr=#1}{#2}
}
\providecommand{\href}[2]{#2}

\end{document}